\documentclass[reqno, a4paper, 12pt]{amsart}
\makeatletter
\let\origsection=\section 
\def\section{\@ifstar{\origsection*}{\mysection}}
\def\mysection{\@startsection{section}{1}\z@{.7\linespacing\@plus\linespacing}{
.5\linespacing}{\normalfont\scshape\centering\S}}
\makeatother

\usepackage{comment}
\usepackage{amsmath,amssymb,amsthm}
\usepackage{mathrsfs}
\usepackage{dsfont}
\usepackage{mathabx}\changenotsign
\usepackage[ruled,linesnumbered,vlined]{algorithm2e}

\usepackage{tikz}
\usetikzlibrary{shapes,decorations,arrows,calc,arrows.meta,fit,positioning}
\usepackage{graphicx}
\usepackage{fullpage}
\usetikzlibrary{decorations.pathreplacing, matrix}

\usepackage{lineno}

\usepackage{microtype}
\usepackage{color}
\usepackage[utf8]{inputenc}
\usepackage{enumitem}

\usepackage{xcolor}
\usepackage{hyperref}
\hypersetup{%
    colorlinks,
    linkcolor={red!60!black},
    citecolor={green!60!black},
    urlcolor={blue!60!black}
}

\usepackage{tikz}
\usepackage{subcaption}
\usetikzlibrary{graphs}

\usepackage{bookmark}

\usepackage[T1]{fontenc}
\usepackage{lmodern}

\usepackage[english]{babel}

\usepackage{fullpage}
\usepackage{setspace}
\usepackage{enumitem}

\let\eps=\varepsilon
\let\polishlcross=\l
\def\l{\ifmmode\ell\else\polishlcross\fi}

\let\emptyset=\varnothing
\let\setminus=\smallsetminus

\makeatletter
\def\moverlay{\mathpalette\mov@rlay}
\def\mov@rlay#1#2{\leavevmode\vtop{%
   \baselineskip\z@skip \lineskiplimit-\maxdimen
   \ialign{\hfil$\m@th#1##$\hfil\cr#2\crcr}}}
\newcommand{\charfusion}[3][\mathord]{
    #1{\ifx#1\mathop\vphantom{#2}\fi
        \mathpalette\mov@rlay{#2\cr#3}
      }
    \ifx#1\mathop\expandafter\displaylimits\fi}
\makeatother

\newcommand{\abs}[1]{\left| #1 \right|}
\newcommand{\cbc}[1]{\left\lbrace #1 \right\rbrace}
\newcommand{\bc}[1]{\left( #1 \right)}
\newcommand{\floor}[1]{\left\lfloor #1 \right\rfloor}

\newtheorem{theorem}{Theorem}
\newtheorem{lemma}[theorem]{Lemma}

\newtheorem{claim}[theorem]{Claim}

\newtheorem{definition}[theorem]{Definition}
\newtheorem{fact}[theorem]{Fact}

\def\cW{\mathcal{W}}
\def\cZ{\mathcal{Z}}
\def\cA{\mathcal{A}}
\def\cB{\mathcal{B}}

\usepackage{datetime}
\usepackage{lineno}
\newcommand*\patchAmsMathEnvironmentForLineno[1]{%
\expandafter\let\csname old#1\expandafter\endcsname\csname #1\endcsname
\expandafter\let\csname oldend#1\expandafter\endcsname\csname end#1\endcsname
\renewenvironment{#1}%
{\linenomath\csname old#1\endcsname}%
{\csname oldend#1\endcsname\endlinenomath}}%
\newcommand*\patchBothAmsMathEnvironmentsForLineno[1]{%
\patchAmsMathEnvironmentForLineno{#1}%
\patchAmsMathEnvironmentForLineno{#1*}}%
\AtBeginDocument{%
\patchBothAmsMathEnvironmentsForLineno{equation}%
\patchBothAmsMathEnvironmentsForLineno{align}%
\patchBothAmsMathEnvironmentsForLineno{flalign}%
\patchBothAmsMathEnvironmentsForLineno{alignat}%
\patchBothAmsMathEnvironmentsForLineno{gather}%
\patchBothAmsMathEnvironmentsForLineno{multline}%
}

\newcommand{\oldqed}{}
\def\endofClaim{\hfill\scalebox{.6}{$\Box$}}

\newenvironment{claimproof}[1][Proof]{
  \renewcommand{\oldqed}{\qedsymbol}
  \renewcommand{\qedsymbol}{\endofClaim}
  \begin{proof}[#1]
}{
  \end{proof}
  \renewcommand{\qedsymbol}{\oldqed}
}

\begin{document}
\onehalfspace
\shortdate
\yyyymmdddate
\settimeformat{ampmtime}
\date{\today, \currenttime}

\title{Minimum degree conditions for containing an $r$-regular $r$-connected subgraph}

\author[M.~Hahn-Klimroth]{Max Hahn-Klimroth} 
\address{hahnklim@math.uni-frankfurt.de, Goethe University Frankfurt, Robert-Mayer-Str. 10, 60235 Frankfurt, Germany }
\thanks{MHK is supported by DFG grant CO 646/5.}

\author[O.~Parczyk]{Olaf Parczyk}
\address{parczyk@mi.fu-berlin.de, FU Berlin, Arnimallee 3, 14195 Berlin, Germany}
\thanks{OP is supported by DFG grant PA 3513/1-1.}

\author[Y.~Person]{Yury Person}
\address{yury.person@tu-ilmenau.de, TU Ilmenau, Weimarer Str. 25, 98684 Ilmenau, Germany} 
\thanks{YP is supported by the Carl Zeiss Foundation and by DFG grant PE 2299/3-1.}

\begin{abstract}
We study  optimal minimum degree conditions when an $n$-vertex graph $G$ contains an $r$-regular $r$-connected subgraph. 
We prove for $r$ fixed and $n$ large the condition to be $\delta(G) \ge \frac{n+r-2}{2}$ when $nr \equiv 0 \pmod 2$. This answers a question of M.~Kriesell.
\end{abstract}


\maketitle


\section{Introduction}

A typical question in extremal graph theory is to determine (asymptotically) optimal minimum degree conditions for a graph $G$ on $n$ vertices to contain a given copy of some spanning graph. 
The classical theorem of Dirac~\cite{dirac} asserts the optimal minimum degree condition to contain a Hamilton cycle to be $\tfrac n2$. There are numerous generalisations of this result to higher connected cycles (powers of Hamilton cycles)~\cite{KSS_Seynmour}, which in turn generalise  the theorems of Corradi and Hajnal~\cite{CH63} and  Hajnal and Szemer\'edi~\cite{HS_erdos} about clique factors in graphs. The most comprehensive result which asymptotically subsumes all of the mentioned results is the bandwidth theorem of B\"ottcher, Schacht and Taraz~\cite{BST09}. This theorem provides a sufficient condition, which asymptotically depends only on the chromatic number of a bounded degree graph with sublinear bandwidth to be contained in a given dense graph. We also refer to the excellent survey~\cite{KO09} by K\"uhn and Osthus for more results.

The present work is motivated by a question of Matthias Kriesell~\cite{MKcomm} about optimal minimum degree condition sufficient to assert the existence of a $4$-regular $4$-connected spanning subgraph. This question in turn was motivated by the work of Bang-Jensen and Kriesell on good acyclic orientations of $4$-regular $4$-connected graphs~\cite{BJK19}.

We answer Kriesell's question by proving the following general result about $r$-connected $r$-regular subgraphs of $G$.
\begin{theorem}
     \label{thm:main}
    For any $r \ge 2$ there exists an $n_0$ such that any $n$-vertex graph $G$ with minimum degree $\delta(G) \ge \frac{n+r-2}{2}$, $n \ge n_0$, and $nr \equiv 0 \pmod 2$ contains a spanning $r$-regular $r$-connected subgraph.   
\end{theorem}
Note that for $r \ge 2$ an $n$-vertex graph $G$ with minimum degree $\delta(G) \ge \frac{n+r-2}{2}$ always is $r$-connected, whereas one can easily come up with examples certifying the optimality of this result (e.g.\ two $K_{(n+r)/2}$'s sharing $r$ vertices).
The theorem above asserts that  there are minimal $r$-connected subgraphs of $G$ which are in fact $r$-regular.
Observe that for $r=2$ this follows immediately from Dirac's theorem~\cite{dirac} with $n_0=3$, as  a Hamilton cycle  is $2$-regular and $2$-connected.
Owing to the use of the regularity lemma the $n_0$ given by Theorem~\ref{thm:main} will be very large.

In the following we briefly introduce some notation and discuss possible candidates for $r$-regular $r$-connected subgraphs that will be found  in $G$ by Theorem~\ref{thm:main}.  
The \emph{$t$-blow-up of a graph $F$} is obtained by replacing every vertex by $t$ vertices and every edge by a complete bipartite graph $K_{t,t}$.
Let $C_n$ be the cycle on $n$ vertices and $P_n$ the $n$-vertex path.
We denote by $C_n(t)$ and $P_n(t)$ the $t$-blow-up of $C_n$ and $P_n$, respectively.
We use a similar definition for odd values of $t$.
We denote by $C_n(t-\tfrac12)$ the $t$-blow-up of $C_n$ for $n$ even, where every other edge only gets a $K_{t,t}$ minus a perfect matching.
Similarly, $P_n(t-\tfrac12)$ is the $t$-blow-up of the $n$-vertex path, where every other edge (starting with the first) only gets a $K_{t,t}$ minus a perfect matching.
We also call these the \emph{$(t-\tfrac 12)$-blow-ups}.
Note that $C_n(t)$ is $2t$-regular and $C_n(t-\tfrac 12)$ is $(2t-1)$-regular.

In most cases in our proof of Theorem~\ref{thm:main} we will be able to find a spanning copy of an $\tfrac r2$-blow-up of a cycle, while allowing other structures with all but a small fraction of vertices in $\tfrac r2$-blow-ups of paths (see Section~\ref{sec:o_constructions} for more details).
However, when $n$ is even and not divisible by $4$, the graph $G$ obtained by taking the disjoint union of two cliques $K_{n/2-2}$ and adding four additional vertices that are connected to all previous $n-4$ vertices cannot contain a copy of $C_n(4)$. 
Finally, observe, that the bandwidth theorem~\cite{BST09} guarantees the asymptotically best minimum degree condition $\tfrac n2 +o(n)$. Thus, Theorem~\ref{thm:main} improves this asymptotic bound to the exact one.

Beyond these blow-ups it would be interesting to study the minimum degree threshold for other spanning structures that can be obtained by identifying vertices or edges of copies of a small graph on a cycle.
In particular, when the small graph is not bipartite, this threshold can depend on its chromatic number or critical chromatic number similarly as when taking disjoint copies (see~\cite{KO09}).

\subsection{Organisation of the paper}
The paper is structured as follows. In Section~\ref{sec:tools} we collect the essential tools (regularity and blow-up lemmas), while Section~\ref{sec:overview} provides a proof overview, which consists of three cases (extremal case I, extremal case II and non-extremal case). These cases are dealt with in the subsequent Sections~\ref{sec:extremal1},~\ref{sec:extremal2} and~\ref{sec:non-extremal}.

\section{Tools and Notation}\label{sec:tools}
For standard graph theoretic definitions we refer to Bollob\'as~\cite{Bolbook98}.  
The main tools are Szemerédi's regularity lemma~\cite{Sze_regularity} and the blow-up lemma by Komlós, Sárközy, and Szemerédi~\cite{KSS_Blowup}.
For this let $G=(V,E)$ be a graph.
For any two sets $A,B \subseteq V$ we denote by $e_G(A,B)$ the number of edges of $G$ with one endpoint in $A$ and one in $B$.
Then the \emph{density} $d(A,B)$ between these sets is $\frac{e(A,B)}{|A||B|}$.

\begin{definition}
    The pair $(A,B)$ is \emph{$\varepsilon$-regular} if for all $X \subseteq A$, $Y \subseteq B$ with $|X| \ge \varepsilon |A|$, $|Y| \ge \varepsilon |B|$ we have $|d(X,Y)-d(A,B)| \le \varepsilon$.    
\end{definition}

The following lemma guarantees that (not too small) induced subgraphs of  $\varepsilon$-regular pairs are still regular (although with a slightly worse parameter).
\begin{lemma}[Slicing lemma]
\label{lem:slicing}
Let $(A,B)$ be an $\eps$-regular pair with $d(A,B)=d$, let $\tfrac 12 \ge \gamma > \eps$, and $A' \subseteq A$ and $B' \subseteq B$ be of size $|A'| \ge \gamma |A|$ and $|B'|\ge \gamma |B|$.
Then $(A',B')$ is $2\eps$-regular pair with $d(A,B) \ge d'$, where $|d-d'|\le \eps$.\qed
\end{lemma}

When working with the regular pairs, one often needs a somewhat stronger concept of super-regularity.
\begin{definition}
    The pair $(A,B)$ is an \emph{$(\varepsilon,\delta)$-super-regular} pair if it is $\varepsilon$-regular and $\deg(a,B) \ge \delta |B|$, $\deg(b,A) \ge \delta |A|$ for all $a \in A$, $b \in B$.
\end{definition}

The next lemma asserts that there every $\varepsilon$-regular pair contains an almost spanning super-regular pair.
\begin{lemma}
\label{lem:superreg}
Let $(A,B)$ be an $\eps$-regular pair with $d(A,B)=d$.
Then there exists $A'\subseteq A$ and $B' \subseteq B$ with $|A'|\ge (1-\eps) |A|$ and $|B'| \ge (1-\eps) |B|$ such that $(A',B')$ is a $(2\eps,d-3\eps)$-super-regular pair. \qed
\end{lemma}

We will use the following degree form of the regularity lemma by Komlós and Simonovits~\cite{koml_simon}.

\begin{lemma}[Regularity lemma, degree version]
    \label{lem:regularity}
    For every $\varepsilon>0$ there exists an integer $M$ such that for any graph $G$ and $d \in [0,1]$ there is a partition of $V(G)$ into $\ell+1 \le M$ clusters $V_0,\dots,V_{\ell}$ and a subgraph $G'$ of $G$ such that
    \begin{enumerate}[label=\upshape(P\arabic*)]
        \item \label{reg:size} $|V_0| \le \varepsilon |V(G)|$ and  $|V_i| = L \le \varepsilon |V(G)|$ for all $1 \le i \le \ell$.
        \item \label{reg:deg} $\deg_{G'}(v) \ge \deg_G(v) - (d+\varepsilon) |V|$ for all $v \in V$.
        \item \label{reg:ind} For $1 \le i \le \ell$ the set $V_i$ is independent in $G'$.
        \item \label{reg:reg} For $1 \le i < j \le \ell$ the pair $(V_i,V_j)$ is $\varepsilon$-regular in $G'$ and has density $0$ or $d$.
    \end{enumerate}
\end{lemma}

The blow-up lemma allows us to embed spanning subgraphs with bounded degree. We will use the following special case deduced from~\cite[Remark~13]{KSS_Blowup}.

\begin{lemma}[Bipartite blow-up lemma]
    \label{lem:blowup}
    For each $d,c>0$ and integer $\Delta$ there exist $\varepsilon>0$, $\alpha>0$ and integer $n_0$ such that the following holds for any $n \ge n_0$.
    Let $H$ be a bipartite graphs on classes $A$ and $B$ with $|A|=|B|=n$ such that $(A,B)$ is a $(\varepsilon,d)$-super-regular pair and let $G$ be a bipartite graph on classes $X$ and $Y$ with $|X|=|Y|=n$ that has maximum degree bounded by $\Delta$.
    Moreover, for any $X' \subseteq X$ and $Y' \subseteq Y$ with $|X'|,|Y'| \le \eps n$ let $A_x \subseteq A$ and $B_y \subseteq B$ for each $x \in X'$ and $y \in Y'$ with $|A_x|,|B_x| \ge cn$.
    Then there exists an embedding of $G$ into $H$ such that all $x \in X'$ and $y \in Y'$ are embedded into $A_x$ and $B_y$, respectively.
\end{lemma}

We remark that in our application $X'$ and $Y'$ will be of constant size and all $A_x$ and all $B_y$ will be the same.

\section{Proof overview}\label{sec:overview}

The proof of Theorem~\ref{thm:main} will be split into three cases.
We now explain this case distinction and then give an overview of the proof for each of these cases.
Let $G$ be a graph with minimum degree $\tfrac{n+r-2}{2}$.
For $\alpha>0$ we call $G$ \emph{$\alpha$-extremal} if there are two sets $A,B \subseteq V(G)$ of size $(\frac{1}{2}-\alpha)n \le |A|,|B| \le \frac{n}{2}$ such that $d(A,B) < \alpha$.
With the help of the regularity lemma we will cover the case that $G$ is not $\alpha$-extremal for any $\tfrac{1}{32}>\alpha>0$ in Section~\ref{sec:non-extremal}.

So we can assume that $G$ is $\alpha$-extremal for some $\alpha>0$.
Using the minimum degree condition in $G$ it is easy to see that the sets $A$ and $B$ have to be almost disjoint or almost the same.
This implies that $G$ contains a large set that is 'almost' independent or it is 'close' to the disjoint union of two cliques $K_{n/2}$.
More precisely, there exists $\alpha'>0$ such that one of the following holds:
Either, there are two disjoint sets $A,B \subseteq V(G)$ with $(\tfrac 12 -\alpha')n \le |A|,|B| \le (\tfrac 12 +\alpha')n$ such that $G[A]$ and $G[B]$ have minimum degree $(\tfrac 12 -3\alpha')n$ and every vertex outside of $A \cup B$ has degree at least $\alpha'n$ into $A$ and $B$ -- this will be the first extremal case treated in Section~\ref{sec:extremal1}.
Or, there is one set $A \subseteq V(G)$ with $(\tfrac12 -\alpha')n \le |A| \le (\tfrac12 +\alpha')n$ is such that any vertex in $A$ has degree at least $(\tfrac12 -3\alpha')n$ into $V(G) \setminus A$ and every vertex outside of $A$ has degree at least $3 \alpha' n$ into $A$ -- this is the second extremal case treated in Section~\ref{sec:extremal2}.

Therefore, when choosing $0<\alpha< \tfrac{1}{32}$ sufficiently small for both extremal cases and the remaining cases will be `non-extremal'.
This implies Theorem~\ref{thm:main}.
In the remainder of this section we sketch the argument for each of the three cases and afterwards explain why our constructions are indeed $r$-connected.

\subsection{Non-Extremal Case}
\label{sec:o_non-extremal}

We would like to find a spanning copy of $C_k(\tfrac r2)$ in $G$, but an obvious necessary condition for this is that $v(G) \equiv 0 \pmod{2 \lceil \tfrac r2 \rceil}$.
If this condition is satisfied, we will succeed, and, otherwise, find a slightly locally modified version.
For the proof we will have constants
\begin{align*}
    \varepsilon \ll \nu \ll d \ll \beta \ll \alpha < \frac{1}{32}
\end{align*}
and $s = \lceil \tfrac r2 \rceil$.
We follow similar arguments as in~\cite{KSS_square}, which can be summarised by the following procedure:
\begin{enumerate}[label=\upshape\bf Step \arabic*]
    \item Apply regularity lemma (Lemma~\ref{lem:regularity}) with $\varepsilon$ and $d$ to obtain a regular partition of $G$.\label{step:regularise}
    \item Find $\ell$ $\varepsilon$-regular pairs $(X_i,Y_i)$ covering all but a small set $V_0$ with $|V_0| \le 20 dn$.\label{step:matching}
    \item For $i=1,\dots,\ell$ connect $Y_i$ to $X_{i+1}$ with the $\tfrac r2$-blow-up of a path that we denote by $P_i$.\label{step:connect}
    \item For $i=1,\dots,\ell$ turn $(X_i,Y_i)$ into an $(\varepsilon,d-\varepsilon)$-super-regular pair with $|X_i|=|Y_i|$, slightly increasing $V_0$ to $|V_0| \le 23 dn$.\label{step:superreg}
    \item Repeatedly take $\nu n$ vertices from $V_0$ and append them to the paths $P_i$.\label{step:absorb}
    \item For $i=1,\dots,\ell$ use blow-up lemma (Lemma~\ref{lem:blowup}) to find a spanning copy of an $\tfrac r2$-blow-up of a path in $(X_i,Y_i)$ connecting $P_{i-1}$ with $P_{i}$.\label{step:spanning}
\end{enumerate}

The index $\ell+1$ corresponds to $1$.
\ref{step:regularise} is natural and for~\ref{step:matching} it is enough to find a large matching in a graph with minimum degree close to $\frac{n}{2}$.
During the performance of~\ref{step:absorb} the degree of some vertices might get too small.
In this case we add them to a set $Q$ that we take care of before the next round.
This terminates as in every execution there are at most $3 \varepsilon n \ll \nu n$ vertices added to $Q$.
Apart from this~\ref{step:absorb} is very similar to~\ref{step:connect}, which we now sketch with more details.

Let $X$, $Y$ be the clusters that we want to connect with the $\tfrac r2$-blow-up of a path $P$.
If there is a cluster $Z$ such that $(X,Z)$ and $(Z,Y)$ are $\varepsilon$-regular pairs with density at least $d$ then we can easily find this path.
Otherwise, let $A$ be the union of all clusters $Z$ such that $(X,Z)$ is an $\varepsilon$-regular pair with density at least $d$ and $B$ the union of all clusters $Z$ for $(Y,Z)$ analogously.
By the minimum degree property in the cluster graph we get $|A|,|B| \ge (\frac{1}{2}-\alpha)n$.
As $G$ is not $\alpha$-extremal we have $d(A,B) > \alpha$.
Therefore, there exist two clusters $Z_1 \in A$ and $Z_2 \in B$ with $d(Z_1,Z_2) \ge \alpha$ and then $(X,Z_1)$, $(Z_1,Z_2)$, and $(Z_2,Y)$ are $\varepsilon$-regular pairs with density at least $d$.
Then it is again easy to find the path that we are interested in by following these three regular pairs.

We have to ensure that the end vertices of the paths always have high degree into the other cluster of the respective super-regular pair, because we want to connect them later and keep them through~\ref{step:superreg}.
Furthermore, we have to ensure that in~\ref{step:absorb} the sizes of the $(\varepsilon,d-\varepsilon)$-super-regular pairs remain balanced.
We will give the details in Section~\ref{sec:non-extremal}.

\subsection{Extremal Case I}
\label{sec:o_extremal1}

In this extremal case we will not use the regularity lemma, but the blow-up lemma will be helpful.
Recall that in this case $G$ is `close' to the union of two disjoint cliques of size roughly $\tfrac n2$ on vertex sets $A$ and $B$.
The main challenge is to find a bridge that connects both these cliques.
It is then easy to find the desired structure using the high degrees.
\begin{enumerate}[label=\upshape\bf Step \arabic*]
\item \label{stepd:bridge} In the case when $r$ is even the bridge will be a matching of size $r$ between $A$ and $B$ such that the end-vertices are well connected on their side.
The odd case is a little more delicate and we will find a matching of size $r+1$ or $r$ depending on the size of $V(G) \setminus \bc{A \cup B}$ and the parity of $A$ and $B$.
\item \label{stepd:absorb} Absorb all vertices that do not not belong to $A$ or $B$ by extending both ends of the path.
We can ensure that the left-over on each side has size divisible by $2r$.
\item \label{stepd:cover} It is easy to see that the left-over on both sides can be split into a super-regular pair and that we can cover both with the $\tfrac r2$-blow-up of a path using Lemma~\ref{lem:blowup}.
\end{enumerate}
If we take care of the end-tuples between each of the steps this gives an $r$-regular $r$-connected path-structure covering $G$.
In Section~\ref{sec:extremal1} we will give the details of the even and odd case separately.

\subsection{Extremal Case II}
\label{sec:o_extremal2}

Again, we will not use the regularity lemma in this part, but the blow-up lemma will still be helpful.
We can assume that we have a partition of $V(G)$ into $A$ and $B$ of size $(\tfrac 12 \pm \alpha)n$ such that between these sets we have minimum degree $\alpha n$ and all but at most $\alpha n$ vertices from $A$ (or $B$) have degree $|B|-\alpha n$ (or $|A|-\alpha n$) into $B$ (or $A$).
W.l.o.g.~assume that $|A|+m=\tfrac 12 n=|B|-m$, where $0 \le m \le \alpha n$.
Note that in $G[B]$ we have minimum degree at least $m+\tfrac{r-2}{2}$.
Let $s=\lceil \tfrac r2 \rceil$.
\begin{enumerate}[label=\upshape\bf Step \arabic*]
    \item \label{stepe:find} If $\Delta(G[B]) \le 2 r \alpha n$ find $m$ copies of $K_{1,s}$, such that all vertices are well connected to the other side. Otherwise, separate the vertices with higher degrees, then find copies of $K_{1,s}$, and afterwards find additional copies of $K_{1,r}$, such that the leaves are well connected.
    \item \label{stepe:absorb} Absorb these copies of $K_{1,s}$ and $K_{1,r}$ into an $r$-regular path-structure and then connect these together into one longer path-structure.
    After removing the path that we constructed we are left with sets $A_1\subseteq A$ and $B_1 \subseteq B$ with $|A_1|=|B_1|$.
    \item \label{stepe:absorb2} Absorb all vertices that do not have large degree to the other side into the path by alternating between both sides.
    After removing these vertices we are left with sets $A_2\subseteq A_1$ and $B_2 \subseteq B_1$ with $|A_2|=|B_2|$ and the property that all vertices have large degree to the other side.
    \item \label{stepe:cover} It is easy to see that $(A_2,B_2)$ is a super-regular pair and that we can cover it with the $\tfrac r2$-blow-up of a path using Lemma~\ref{lem:blowup}.
\end{enumerate}
If we take care of the end-tuples between each of the steps this gives an $r$-regular $r$-connected path-structure covering $G$.
For the first step we use the following.
\begin{lemma}
\label{lem:stars}
    For any integer $s$ there exists $\alpha>0$ such that the following holds.
    Let $G$ be an $n$ vertex graph with maximum degree $\Delta(G) \le 4 s \alpha n$ and minimum degree $\delta(G) \ge m + s - 1$, where $1 \le m \le \alpha n$.
    Then there are $2m$ pairwise disjoint copies of $K_{1,s}$ in $G$.
\end{lemma}

The proof of this lemma and the second extremal case will be given in Section~\ref{sec:extremal2}.

\subsection{Constructions}
\label{sec:o_constructions}
First recall that the $\tfrac r2$-blow-up of a cycle is $r$-regular and also $r$-connected.
It will not always be possible to construct this, but it will be the basic building block.
We might need to absorb some exceptional vertices, for example, when $n$ is not divisible by $r$.
In the case when $r$ is even we then remove a perfect matching from one $K_{s,s}$ and add one vertex that is connected to all $2s=r$ vertices that just lost one neighbour (c.f.~Figures~\ref{fig:absorber_even},~\ref{fig:absorbK1s}, and~\ref{fig:Step54}).
The resulting graph is still $r$-connected, because we can not disconnect this part of the cycle by removing less than $\tfrac r2$ vertices.
A similar construction will be used in the case when $r$ is odd (c.f.~Figures~\ref{fig:absorber_odd},~\ref{fig:absorbK1s}~and~\ref{fig:Step55}) that also preserves $r$-connectivity.
Apart from this, we also have to connect to $\tfrac r2$-blow-ups of cycles by using at most $r$ edges between them (c.f.~Step~\ref{stepd:bridge} of Section~\ref{sec:o_extremal1}).
We will only need to take care of a small linear fraction of the vertices from $G$ and, therefore, almost all vertices are in the $\tfrac r2$-blow-up of a path.

\section{Extremal Case I}
\label{sec:extremal1}
In this section we deal with the first extremal case.
We will not use the regularity lemma in this part, but the blow-up lemma will still be helpful.

\begin{proof}[Proof of Extremal Case I]
Let $r \geq 3$ be an integer, let $\eps>0$ be given by Lemma~\ref{lem:blowup} on input $\tfrac 12$, $\tfrac 12$, and $r$ and let $0<\alpha \le \eps (1000 r^2)^{-1}$.
Let $G$ be an $n$-vertex graph with $\delta(G) \geq \tfrac {n+r-2}{2}$ and let $A,B \subseteq V(G)$ with $(\tfrac 12-\alpha)n \le |A|,|B| \le (\tfrac 12+\alpha)n$ such that $G[A]$ and $G[B]$ have minimum degree $(\tfrac 12-3\alpha)n$ and every vertex in $C=V(G) \setminus (A \cup B)$ has degree at least $\alpha n$ into $A$ and $B$.
Our goal is to find an $r$-regular, $r$-connected spanning subgraph in $G$ provided that $n$ is large enough.  

\subsection{The even case}

Assume that $r$ is even.
We begin by constructing $\tfrac r2$ bridges of size 2 between $A$ and $B$ (\ref{stepd:bridge} of Section \ref{sec:o_extremal1}). A visualisation can be found in Figure~\ref{fig:bridge_graph_h}.

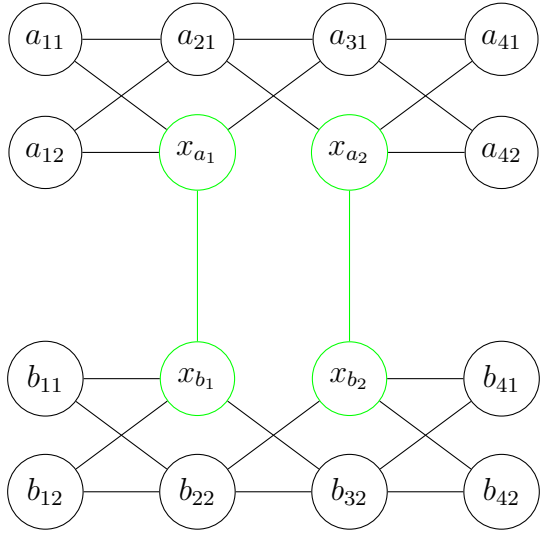
\begin{figure}
    \centering
    \begin{tikzpicture}
        \node[circle, draw=black] (a_11) at (0, 3) {$a_{11}$};
        \node[circle, draw=black] (a_12) at (0, 1.5) {$a_{12}$};
        \node[circle, draw=black] (b_11) at (0, -1.5) {$b_{11}$};
        \node[circle, draw=black] (b_12) at (0, -3) {$b_{12}$};
        
        \node[circle, draw=black] (a_21) at (2, 3) {$a_{21}$};
        \node[circle, draw=green] (xa1) at (2, 1.5) {$x_{a_1}$};
        \node[circle, draw=green] (xb1) at (2, -1.5) {$x_{b_1}$};
        \node[circle, draw=black] (b_22) at (2, -3) {$b_{22}$};
        
        \node[circle, draw=black] (a_31) at (4, 3) {$a_{31}$};
        \node[circle, draw=green] (xa2) at (4, 1.5) {$x_{a_2}$};
        \node[circle, draw=green] (xb2) at (4, -1.5) {$x_{b_2}$};
        \node[circle, draw=black] (b_32) at (4, -3) {$b_{32}$};
        
        \node[circle, draw=black] (a_41) at (6, 3) {$a_{41}$};
        \node[circle, draw=black] (a_42) at (6, 1.5) {$a_{42}$};
        \node[circle, draw=black] (b_41) at (6, -1.5) {$b_{41}$};
        \node[circle, draw=black] (b_42) at (6, -3) {$b_{42}$};

        \path[-] (a_11) edge[draw=black] (a_21);
        \path[-] (a_12) edge[draw=black] (xa1);
        \path[-] (a_11) edge[draw=black] (xa1);
        \path[-] (a_12) edge[draw=black] (a_21);
        
        \path[-] (xa1) edge[draw=black] (a_31);
        \path[-] (a_21) edge[draw=black] (xa2);
        \path[-] (a_21) edge[draw=black] (a_31);
        
        \path[-] (a_31) edge[draw=black] (a_41);
        \path[-] (xa2) edge[draw=black] (a_42);
        \path[-] (a_31) edge[draw=black] (a_42);
        \path[-] (xa2) edge[draw=black] (a_41);

        \path[-] (b_11) edge[draw=black] (xb1);
        \path[-] (b_12) edge[draw=black] (b_22);
        \path[-] (b_11) edge[draw=black] (b_22);
        \path[-] (b_12) edge[draw=black] (xb1);
        
        \path[-] (b_22) edge[draw=black] (xb2);
        \path[-] (xb1) edge[draw=black] (b_32);
        \path[-] (b_22) edge[draw=black] (b_32);
        
        \path[-] (xb2) edge[draw=black] (b_41);
        \path[-] (b_32) edge[draw=black] (b_42);
        \path[-] (xb2) edge[draw=black] (b_42);
        \path[-] (b_32) edge[draw=black] (b_41);
        
        \path[-] (xb2) edge[draw=green] (xa2);
        \path[-] (xb1) edge[draw=green] (xa1);

    \end{tikzpicture}
    \caption{Bridge between the sets $A$ and $B$ in the special case $r = 4$.}
    \label{fig:bridge_graph_h}
\end{figure}

\begin{claim}\label{lemma_matching_edges}
Suppose $\delta(G) \geq \tfrac {n+r-2}{2}$ and $\abs{A} \leq \abs{B}$. There is a matching $(x_{a_1}x_{b_1}, \ldots, x_{a_r}x_{b_r})$  such that $\abs{N(x_{a_i}) \cap A} \geq \tfrac n5$ and $\abs{N(x_{b_j}) \cap B} \geq \tfrac n5$ for all $i, j \leq r$.
\end{claim}
\begin{claimproof}
 In order to construct the matching, it suffices to find $r$ edges from $A$ to $V \setminus A$. Indeed, suppose we find the $r$ edges $a_1c_1, \ldots, a_rc_r$. If $c_i \in B$, we take this edge. If $c_i \in V \setminus (A \cup B)$, $c_i$ has either $\tfrac n5$ edges into $B$ (in this case, take edge $a_ic_i$), or it has $\tfrac n5$ edges into $A$. Let $i_1, \ldots, i_l$ be the indices such that $c_{i_j}$ does not have $\tfrac n5$ edges into $B$. By definition of $\alpha$-extremity, each $c_{i_j}$ has $\alpha n$ neighbors in $B$. Select $b_{i_1}, \ldots, b_{i_l}$ s.t. $b_{i_j} \in N(c_{i_j}) \cap B$ and $b_{i_j} \neq b_{i_k}$ for all $j \neq k$ (and being disjoint from those $c_i \in B$ (this is clearly possible). We add edges $b_{i_j}c_{i_j}$ to the matching. 
 
It remains to show that these edges exist. 
First, suppose that $n$ is even. If $\abs{A} \leq \tfrac{n-r}{2}$, the minimum degree of $ \tfrac{n+r-2}{2}$ guarantees that each vertex of $A$ needs to find at least $\tfrac{n+r-2}{2} - (\tfrac{n-r}{2} - 1) = r$ neighbors outside of $A$, hence the assertion follows.
Suppose $\abs{A} = \tfrac{n-r}{2} + i$ with $i = 1, \ldots, \tfrac r2$. In this case, $\abs{V \setminus A} = \tfrac{n+r}{2} - i$. Each vertex of $A$ finds at least $r-i$ neighbors outside of $A$. Suppose that $N(A) \setminus A$ has size at most $r-1$ (thus, all edges from $A$ into the rest of the graph belong to $r-1$ vertices). Now pick a different vertex in the complement (which exists, as $\tfrac{n+r}{2} - i \gg r$). This vertex requires $\tfrac{n+r-2}{2} - (\tfrac{n+r}{2} - i - 1) = i$ neighbors in $A$, which is a contradiction.

Now, if $n$ is odd, because the minimum degree needs to be an integer, it is at least $\tfrac{n+1+r-2}{2}$, hence upon removal of one vertex, we are left with a graph on $n' = n-1$ vertices and minimum degree at least $\tfrac{n+1+r-2}{2}-1 = \tfrac{n'+r-2}{2}$ (and $n'$ being even). Hence the assertion follows from the previous discussion.
\end{claimproof}

Therefore, Claim~\ref{lemma_matching_edges} gives us the green sub-structure of Figure \ref{fig:bridge_graph_h}. 
Now, we take two of those matching edges (think of them as being $\tfrac r2$ pairs of $2$ edges). Denote the vertices that are connected to at least $\tfrac n5$ vertices in $A$ as $x_{a_1}, x_{a_2}$. We next prove that the black structure around $x_{a_1}, x_{a_2}$ shown Figure \ref{fig:bridge_graph_h} exists.

\begin{claim}\label{bridge_graph_existence}
There are distinct vertices $a_{i,1}, \ldots, a_{i, r/2} \in A$ for $i=1,4$ and $a_{i,1}, \ldots, a_{i, r/2-1} \in A$ for $i=2,3$ with the following properties.
\begin{enumerate}
    \item The edges $a_{i, j}a_{i+1, k}$ for $i=1, 2 ,3$ and $j = 1 , \ldots, \tfrac r2$ (or $\tfrac r2 - 1$, respectively) exist,
    \item the edges $x_{a_1}a_{1, j}$ and $x_{a_2}a_{4, j}$ exist for $j = 1 , \ldots,\tfrac r2$,
    \item the edges $x_{a_1}a_{3, j}$ and $x_{a_2}a_{1, j}$ exist for $j = 1 , \ldots,\tfrac r2 -1$.
\end{enumerate}
\end{claim}

\begin{claimproof}
We select $r - 1$ vertices $a_{1,i}$, $a_{3, j} \in N(x_{a_1}) \cap A$ arbitrarily (but disjoint from $x_{a_2}$). Those exist as $x_{a_1}$ has at least $\tfrac n5$ neighbors in $A$. Each of those vertices is connected to at least $(\tfrac 12 - 3 \alpha)n$ vertices in $A$, hence each vertex has at least $(\tfrac 12 - 3 \alpha)n - r - 1$ neighbors in $A$ that do not belong to $a_{1,i}$, $a_{3, j}$ or $x_{a_1}, x_{a_2}$.
Therefore, the joint neighborhood 
\[N := \bc{A \cap \bigcap_{i = 1}^{r/2} N(a_{1,i}) \bigcap_{j=1}^{r/2 - 1} N(a_{1,j})} \setminus \bc{\bigcup_{i=1}^{r/2} \cbc{ a_{4, i} } \bigcup_{j=1}^{r/2-1} a_{3, j} \cup \cbc{ x_{a_1}, x_{a_2} }}\]
has size at least $(\tfrac 12 - 4 r \alpha)n$. Therefore, we find 
\[\abs{ N(x_{a_2}) \cap N } \geq \frac{n}{100} \, , \]
thus the claim follows as the same token holds in $B$ as well.
\end{claimproof}
We denote the resulting collection of vertices in $A$ by $X_{A, 1}, \ldots, X_{A, r/2}$.
Clearly, each vertex in $A$ stays connected to at least $(\tfrac 12 - 4 \alpha)n$ vertices in
\[A'_{0} =  A \setminus \bc{ V(X_{A, 1}) \cup \ldots \cup V(X_{A, r/2}) } \, .\]

Next, we introduce a \textit{gluing operation} GE.
\begin{claim}[Gluing operation GE]
Given two disjoint sets $D_1, D_2 \subset A'_0$ of size exactly $\tfrac r2$, we find two disjoint sets $D, D' \subset A'_0 \setminus (D_1 \cup D_2)$ of size $\tfrac r2$ such that
\[G[D_1, D] \equiv K_{r/2, r/2}, \quad G[D, D'] \equiv K_{r/2, r/2}\quad \text{and} \quad G[D', D_2] \equiv K_{r/2, r/2} \, .\]
\end{claim}
\begin{claimproof}
As the joint $A'_0$ - neighborhood of $D_1$ and $D_2$ has size at least $(\tfrac 12 - 10r \alpha)n$, the assertion follows.
\end{claimproof}

Using GE, we glue the $\tfrac r2$ bridges in $A$ and $B$ respectively together using mutually disjoint vertex sets $D^A_1, \ldots,D^A_{r/2 - 1}$ and $D^B_1, \ldots,D^B_{r/2 - 1}$ and are left with path-like structures $P_A$ and $P_B$. After gluing the bridges together, we let $A' = A'_0 \setminus V(P_A), B' = B'_0 \setminus V(P_B)$. 

In a next step we need to absorb left-over vertices (\ref{stepd:absorb} of Section \ref{sec:o_extremal1}). To this end define two absorber-graphs for a vertex $u$: $\xi_r(u)$ and $\xi'_r(u)$ (see Figure \ref{fig:absorber_even}).
\begin{definition}
    Let $D \in \cbc{A', B'}$ and $u$ a vertex such that $\abs{ N(u) \cap D } \geq \tfrac n6$. Define $\xi_r(u)$ as follows.
    \begin{itemize}
        \item Select $D_1 = \cbc{d_1, d_2, \ldots, d_{r/2}}, D_2 = \cbc{d'_1, \ldots, d'_{r/2}} \subset N(u) \cap D$, hence $r$ pairwise disjoint vertices.
        \item Select $D' = \cbc{ u'_1, \ldots, u'_{r/2 - 1}} \subset N(D_1) \cap N(D_2) \cap D \setminus (D_1 \cup D_2 \cup \cbc{u}) $.
    \end{itemize}
    Define $\xi_r(u)$ as the graph containing $D_1, D_2, D'$ and $u$ as well as all the edges from $D_1$ to $D' \cup \cbc{u}$ and from $D_2$ to $D' \cup \cbc{u}$.
    Furthermore, define $\xi'_r(u)$ via
    \begin{itemize}
        \item Select $D_1 = \cbc{d_1, d_2, \ldots, d_{r/2}}, D_2 = \cbc{d'_1, \ldots, d'_{r/2}} \subset N(u) \cap D$, hence $r$ pairwise disjoint vertices.
        \item Select $D' = \cbc{ u'_1, \ldots, u'_{r/2 - 1}} \subset N(D_1) \cap N(D_2) \cap D \setminus (D_1 \cup D_2 \cup \cbc{u}) $.
        \item Select $\tfrac r2$ vertices $E_0 = \cbc{e_0, \ldots, e_{r/2-1}}$ and an additional disjoint vertex $e_{r/2}$ from $D \setminus \bc{ D_1 \cup D_2 \cup D' }$ such that $E_0 \cup \cbc{e_2} \subset N(d_2) \cap \ldots \cap N(d_{r/2}) \cap D$ and $G[E_0, \cbc{e_{r/2}}] \equiv K_{r/2, 1}$. 
    \end{itemize}
    Define $\xi'_r(u)$ as the graph containing $D_1, D_2, D'$ and $u$ as well as all the edges from $D_1$ to $D' \cup \cbc{u}$ and from $D_2$ to $D' \cup \cbc{u}$.
\end{definition}
Clearly, by the sizes of $A', B'$ and the minimum-degree condition, given at most $100 \alpha n$ pairwise different vertices, there is an absorber for each vertex which is disjoint from all other absorbers. Furthermore, by the minimum degree condition inside of $A'$ and $B'$, this family of absorbers exists.
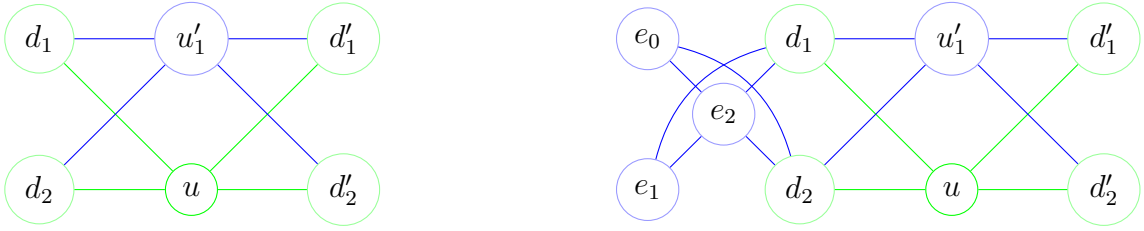
\begin{figure}
    \centering
    \begin{tikzpicture}

     \node[circle, draw=green!40] (d1) at (2, 1) {$d_1$};
     \node[circle, draw=green!40] (d2) at (2, -1) {$d_2$};

     \node[circle, draw=green!40] (dp1) at (6, 1) {$d'_{1}$};
     \node[circle, draw=green!40] (dp2) at (6, -1) {$d'_{2}$};
     
     \node[circle, draw=blue!40] (up1) at (4, 1) {$u'_{1}$};
     \node[circle, draw=green] (u) at (4, -1) {$u$};

     \path[-] (u) edge[draw=green] (d1);
     \path[-] (u) edge[draw=green] (d2);
     \path[-] (u) edge[draw=green] (dp1);
     \path[-] (u) edge[draw=green] (dp2);
     
     \path[-] (d1) edge[draw=blue] (up1);
     \path[-] (d2) edge[draw=blue] (up1);
     \path[-] (dp1) edge[draw=blue] (up1);
     \path[-] (dp2) edge[draw=blue] (up1);
     
     
     \node[circle, draw=blue!40] (e0) at (10, 1) {$e_0$};
     \node[circle, draw=blue!40] (e1) at (10, -1) {$e_1$};
     \node[circle, draw=blue!40] (e2) at (11, 0) {$e_2$};
     
     \node[circle, draw=green!40] (ld1) at (12, 1) {$d_1$};
     \node[circle, draw=green!40] (ld2) at (12, -1) {$d_2$};

     \node[circle, draw=green!40] (ldp1) at (16, 1) {$d'_{1}$};
     \node[circle, draw=green!40] (ldp2) at (16, -1) {$d'_{2}$};
     
     \node[circle, draw=blue!40] (lup1) at (14, 1) {$u'_{1}$};
     \node[circle, draw=green] (lu) at (14, -1) {$u$};

     \path[-] (lu) edge[draw=green] (ld1);
     \path[-] (lu) edge[draw=green] (ld2);
     \path[-] (lu) edge[draw=green] (ldp1);
     \path[-] (lu) edge[draw=green] (ldp2);
     
     \path[-] (ld1) edge[draw=blue] (lup1);
     \path[-] (ld2) edge[draw=blue] (lup1);
     \path[-] (ldp1) edge[draw=blue] (lup1);
     \path[-] (ldp2) edge[draw=blue] (lup1);
     
     \path[-] (ld1) edge[draw=blue] (e2);
     \path[-] (ld2) edge[draw=blue] (e2);
     \path[-] (e1) edge[draw=blue] (e2);
     \path[-] (e0) edge[draw=blue] (e2);
     
     \path[-] (e0) edge[draw=blue, bend left=30] (ld2);
     \path[-] (e1) edge[draw=blue, bend left=30] (ld1);
    
    \end{tikzpicture}
    
    \caption{Absorbers $\xi_4(u)$ (left) and $\xi'_4(u)$ (right) with $r=4$, where $u$ is the green vertex, the green vertices are inside the $A'$-neighborhood of $u$ and the blue vertices are vertices chosen from $A'$.}
    \label{fig:absorber_even}
\end{figure}

We are now in position to absorb the exceptional set $C$ (of course, without the bridging vertices on $P_A$ and $P_B$. For all up to $6 \alpha n$ vertices in $C$ create a disjoint absorber $\xi_r(u)$ as above. Chose as $D$ either $A'$ (if a vertex has $\tfrac n5$ neighbors in $A$), or $B'$ otherwise. Next, using the gluing operation GE up to $6 \alpha n$ times, glue the absorbers inside of $A'$ and $B'$ together, always using only vertices that did not get used in a previous gluing step or are part of the absorbers. As one only requires at most $24 r \alpha n$ vertices during this procedure, it is clearly possible by above discussion.

Finally, use GE again to glue the series of absorbers to $P_A$ and $P_B$ respectively (which is clearly possible, as this is only one operation on each set).

After the repetitive gluing, we are left with sets $A''$ and $B''$ (hence, $A'$ without the glued structures $P'_A$) and the path-like subgraph $P'_A$ of size at most $25r \alpha n$, hence $\abs{A''} \geq (\tfrac 12 - 30 r \alpha)n$. Furthermore, each vertex in $A''$ is connected to at least $(\tfrac 12 - 40 r \alpha)n$ vertices in $A''$. Clearly, the same holds for $B''$. 

Before closing the path in both sets (hence, creating a cycle which contains all vertices that are not part of $P'_A$ or $P'_B$), which is a standard application of the blow-up lemma, we need to make sure that certain divisibility conditions hold. 
As we wish to close the cycle by appending blocks of two layers of size $r$ (thus, $K_{{r/2}, {r/2}}$), we require that $\abs{A''} \equiv \abs{B''} \equiv 0 \pmod r$. If this is the case, set $A''' = A''$ and proceed. Otherwise, if there is $0 < i < r$ such that $\abs{A''} \equiv i \pmod r$, select $i$ vertices $a_1, \ldots,a_{i} \in A''$ and absorb them using disjoint instances $\xi'_r(a_1), \ldots, \xi'_r(a_{i})$ with $D = A''$. Clearly, as this requires only finitely many vertices, such a disjoint family exists. Further, because $a_j \in A''$, each absorber consumes $2r + 1$ vertices of $A''$, hence afterwards, the divisibility condition holds. Now, glue the absorbers sequentially to $P'_A$ using GE and sets $G_1, \ldots,G_{i-1}$. As each gluing operation consumes $r$ vertices, the divisibility does not change hence we are left with a set $A''' = A'' \setminus \bc{ V(\xi'_r(a_1)) \cup \ldots \cup V(\xi'_r(a_{i})) \cup G_1 \cup \ldots \cup G_{i-1} }$.

Now it is easy to check that $(A''',B''')$ is $(\eps,\tfrac 12)$-super-regular and by Lemma~\ref{lem:blowup} we find an $\tfrac r2-$blowup of the path on all remaining vertices of $A'''$ and $B'''$ (\ref{stepd:cover} of Section \ref{sec:o_extremal1}). 
Moreover, the end-tuples of the path-like structure constructed before have at least $\tfrac 12 |A'''|$ and $\tfrac 12 |B'''|$ common neighbours in $A'''$ and $B'''$ respectively.
Therefore, we may choose the start- and end-tuples of this path-blow-up to connect to these end-tuples.

We are left to argue that the constructed subgraph is $r$-connected and $r$-regular.
\begin{claim}
The constructed subgraph is $r$-connected and $r$-regular.
\end{claim}
\begin{claimproof}
While $r$-regularity follows obviously, the $r$-connected part needs a short argument. Upon removal of up to $r-1$ bridge-vertices, the parts do not fall apart. Furthermore, removing up to $r-1$ vertices in the $\tfrac r2$-blow-up of the path part of the subgraph does not disconnect the structure. Finally, the absorbing structure $\xi_r$ itself is isomorphic to an $\tfrac r2$-blowup of the path on three vertices. Moreover, disconnecting the graph by removing up to $r-1$ vertices in $\xi_r'$ is not possible.
\end{claimproof}

\subsection{The odd case}

Assume that $r$ is odd.
The argument in the odd case is a bit more delicate as in the even case. Indeed, while in the process above all divisibility conditions could be easily established, in the odd case, we might end with two almost cliques of odd size. If there is a set $C$, we can easily absorb those vertices in a way that after absorbing both parts of the graph contain an even number of vertices - which we require to embed a regular graph. If on the other hand there is no such set $C$, we need to be much more careful. We will tackle this problem by having two different types of bridges between $A$ and $B$, one consuming an even number of vertices of each set, one consuming an odd number - thus, depending on the size of $C$ and the parity of $A$ and $B$, we need to use two different constructions. The two types of bridges are visualised in Figure~\ref{fig:bridge_graph_h_odd} for the special case $ r = 5$.

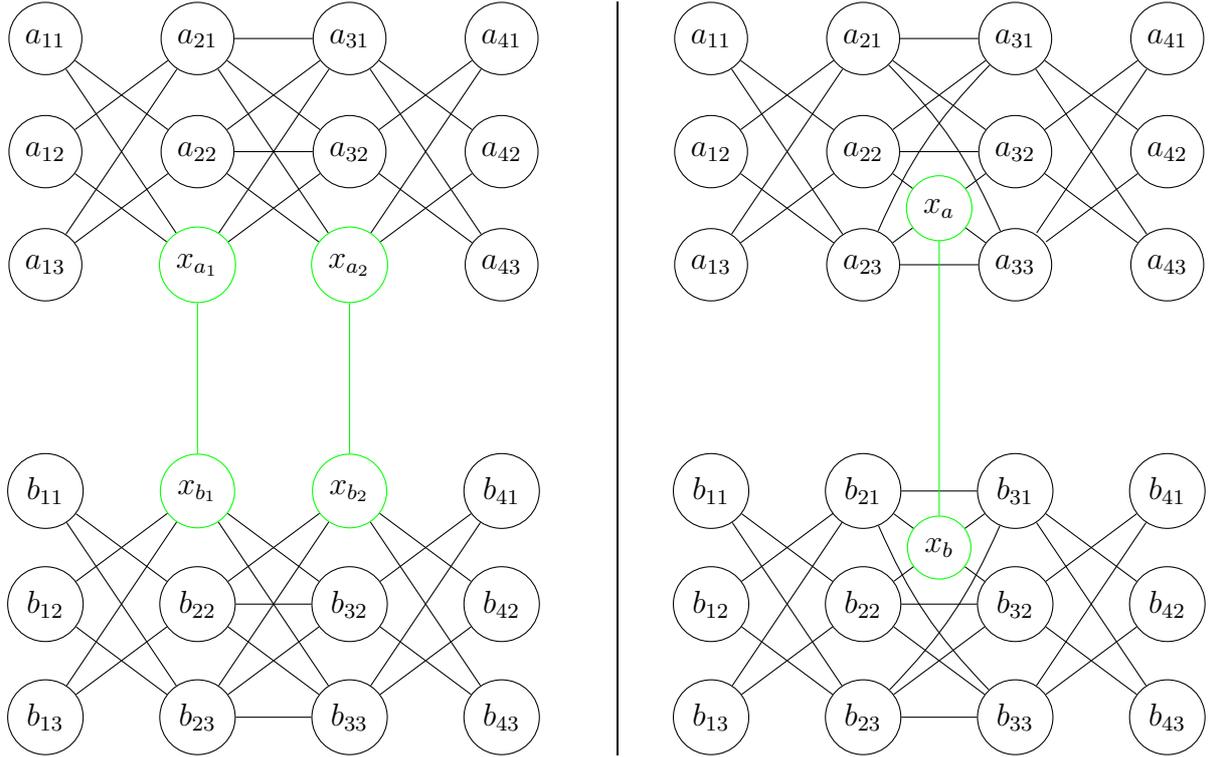
\begin{figure}
    \centering
    \begin{minipage}{0.45\textwidth}
    \begin{tikzpicture}
        \node[circle, draw=black] (a_11) at (0, 4.5) {$a_{11}$};
        \node[circle, draw=black] (a_12) at (0, 3) {$a_{12}$};
        \node[circle, draw=black] (a_13) at (0, 1.5) {$a_{13}$};
        \node[circle, draw=black] (b_11) at (0, -1.5) {$b_{11}$};
        \node[circle, draw=black] (b_12) at (0, -3) {$b_{12}$};
        \node[circle, draw=black] (b_13) at (0, -4.5) {$b_{13}$};
        
        \node[circle, draw=black] (a_21) at (2, 4.5) {$a_{21}$};
        \node[circle, draw=black] (a_22) at (2, 3) {$a_{22}$};
        \node[circle, draw=green] (xa1) at (2, 1.5) {$x_{a_1}$};
        \node[circle, draw=green] (xb1) at (2, -1.5) {$x_{b_1}$};
        \node[circle, draw=black] (b_22) at (2, -3) {$b_{22}$};
        \node[circle, draw=black] (b_23) at (2, -4.5) {$b_{23}$};
        
        \node[circle, draw=black] (a_31) at (4, 4.5) {$a_{31}$};
        \node[circle, draw=black] (a_32) at (4, 3) {$a_{32}$};
        \node[circle, draw=green] (xa2) at (4, 1.5) {$x_{a_2}$};
        \node[circle, draw=green] (xb2) at (4, -1.5) {$x_{b_2}$};
        \node[circle, draw=black] (b_32) at (4, -3) {$b_{32}$};
        \node[circle, draw=black] (b_33) at (4, -4.5) {$b_{33}$};
        
        \node[circle, draw=black] (a_41) at (6, 4.5) {$a_{41}$};
        \node[circle, draw=black] (a_42) at (6, 3) {$a_{42}$};
        \node[circle, draw=black] (a_43) at (6, 1.5) {$a_{43}$};
        \node[circle, draw=black] (b_41) at (6, -1.5) {$b_{41}$};
        \node[circle, draw=black] (b_42) at (6, -3) {$b_{42}$};
        \node[circle, draw=black] (b_43) at (6, -4.5) {$b_{43}$};

        \path[-] (a_11) edge[draw=black] (a_22);
        \path[-] (a_11) edge[draw=black] (xa1);
        \path[-] (a_12) edge[draw=black] (xa1);
        \path[-] (a_12) edge[draw=black] (a_21);
        \path[-] (a_13) edge[draw=black] (a_21);
        \path[-] (a_13) edge[draw=black] (a_22);

        \path[-] (xa1) edge[draw=black] (a_31);
        \path[-] (xa1) edge[draw=black] (a_32);
        \path[-] (a_21) edge[draw=black] (xa2);
        \path[-] (a_21) edge[draw=black] (a_31);
        \path[-] (a_21) edge[draw=black] (a_32);
        \path[-] (a_22) edge[draw=black] (xa2);
        \path[-] (a_22) edge[draw=black] (a_31);
        \path[-] (a_22) edge[draw=black] (a_32);
        
        \path[-] (a_31) edge[draw=black] (a_42);
        \path[-] (a_31) edge[draw=black] (a_43);
        \path[-] (a_32) edge[draw=black] (a_41);
        \path[-] (a_32) edge[draw=black] (a_43);
        \path[-] (xa2) edge[draw=black] (a_41);
        \path[-] (xa2) edge[draw=black] (a_42);

        \path[-] (b_12) edge[draw=black] (b_23);
        \path[-] (b_11) edge[draw=black] (b_22);
        \path[-] (b_11) edge[draw=black] (b_23);
        \path[-] (b_12) edge[draw=black] (xb1);
        \path[-] (b_13) edge[draw=black] (b_22);
        \path[-] (b_13) edge[draw=black] (xb1);

        \path[-] (b_22) edge[draw=black] (xb2);
        \path[-] (xb1) edge[draw=black] (b_32);
        \path[-] (xb1) edge[draw=black] (b_33);
        \path[-] (b_22) edge[draw=black] (b_32);
        \path[-] (b_22) edge[draw=black] (b_33);
        \path[-] (b_23) edge[draw=black] (b_32);
        \path[-] (b_23) edge[draw=black] (xb2);
        \path[-] (b_23) edge[draw=black] (b_33);
        
        \path[-] (xb2) edge[draw=black] (b_43);
        \path[-] (b_32) edge[draw=black] (b_43);
        \path[-] (xb2) edge[draw=black] (b_42);
        \path[-] (b_32) edge[draw=black] (b_41);
        \path[-] (b_33) edge[draw=black] (b_41);
        \path[-] (b_33) edge[draw=black] (b_42);
        
        \path[-] (xb2) edge[draw=green] (xa2);
        \path[-] (xb1) edge[draw=green] (xa1);

    \end{tikzpicture}
    \end{minipage} \hfill\vline\hfill
    \begin{minipage}{0.45\textwidth}
        \begin{tikzpicture}
        \node[circle, draw=black] (a_11) at (0, 4.5) {$a_{11}$};
        \node[circle, draw=black] (a_12) at (0, 3) {$a_{12}$};
        \node[circle, draw=black] (a_13) at (0, 1.5) {$a_{13}$};
        \node[circle, draw=black] (b_11) at (0, -1.5) {$b_{11}$};
        \node[circle, draw=black] (b_12) at (0, -3) {$b_{12}$};
        \node[circle, draw=black] (b_13) at (0, -4.5) {$b_{13}$};
        
        \node[circle, draw=black] (a_21) at (2, 4.5) {$a_{21}$};
        \node[circle, draw=black] (a_22) at (2, 3) {$a_{22}$};
        \node[circle, draw=black] (a_23) at (2, 1.5) {$a_{23}$};
        \node[circle, draw=black] (b_21) at (2, -1.5) {$b_{21}$};
        \node[circle, draw=black] (b_22) at (2, -3) {$b_{22}$};
        \node[circle, draw=black] (b_23) at (2, -4.5) {$b_{23}$};
        
        \node[circle, draw=black] (a_31) at (4, 4.5) {$a_{31}$};
        \node[circle, draw=black] (a_32) at (4, 3) {$a_{32}$};
        \node[circle, draw=black] (a_33) at (4, 1.5) {$a_{33}$};
        \node[circle, draw=black] (b_31) at (4, -1.5) {$b_{31}$};
        \node[circle, draw=black] (b_32) at (4, -3) {$b_{32}$};
        \node[circle, draw=black] (b_33) at (4, -4.5) {$b_{33}$};
        
        \node[circle, draw=black] (a_41) at (6, 4.5) {$a_{41}$};
        \node[circle, draw=black] (a_42) at (6, 3) {$a_{42}$};
        \node[circle, draw=black] (a_43) at (6, 1.5) {$a_{43}$};
        \node[circle, draw=black] (b_41) at (6, -1.5) {$b_{41}$};
        \node[circle, draw=black] (b_42) at (6, -3) {$b_{42}$};
        \node[circle, draw=black] (b_43) at (6, -4.5) {$b_{43}$};
        
        \node[circle, draw=green] (xa) at (3, 2.25) {$x_{a}$};
        \node[circle, draw=green] (xb) at (3, -2.25) {$x_{b}$};

        \path[-] (a_12) edge[draw=black] (a_23);
        \path[-] (a_11) edge[draw=black] (a_23);
        \path[-] (a_11) edge[draw=black] (a_22);
        \path[-] (a_12) edge[draw=black] (a_21);
        \path[-] (a_13) edge[draw=black] (a_21);
        \path[-] (a_13) edge[draw=black] (a_22);

        \path[-] (a_21) edge[draw=black] (a_31);
        \path[-] (a_21) edge[draw=black] (a_32);
        \path[-] (a_22) edge[draw=black] (a_31);
        \path[-] (a_22) edge[draw=black] (a_32);
        
        \path[-] (a_31) edge[draw=black] (a_42);
        \path[-] (a_31) edge[draw=black] (a_43);
        \path[-] (a_32) edge[draw=black] (a_41);
        \path[-] (a_32) edge[draw=black] (a_43);
        \path[-] (xa2) edge[draw=black] (a_41);
        \path[-] (xa2) edge[draw=black] (a_42);

        \path[-] (b_12) edge[draw=black] (b_23);
        \path[-] (b_11) edge[draw=black] (b_22);
        \path[-] (b_11) edge[draw=black] (b_23);
        \path[-] (b_12) edge[draw=black] (xb1);
        \path[-] (b_13) edge[draw=black] (b_22);
        \path[-] (b_13) edge[draw=black] (xb1);

        \path[-] (b_22) edge[draw=black] (b_32);
        \path[-] (b_22) edge[draw=black] (b_33);
        \path[-] (b_23) edge[draw=black] (b_32);
        \path[-] (b_23) edge[draw=black] (b_33);
        
        \path[-] (b_32) edge[draw=black] (b_43);
        \path[-] (b_32) edge[draw=black] (b_41);
        \path[-] (b_33) edge[draw=black] (b_41);
        \path[-] (b_33) edge[draw=black] (b_42);
        
        \path[-] (b_31) edge[draw=black] (b_42);
        \path[-] (b_31) edge[draw=black] (b_43);
        
        \path[-] (b_21) edge[draw=black, bend right=10] (b_33);
        \path[-] (b_31) edge[draw=black, bend left=10] (b_23);
        \path[-] (b_31) edge[draw=black] (b_21);
        
        \path[-] (xb) edge[draw=black] (b_21);
        \path[-] (xb) edge[draw=black] (b_22);
        \path[-] (xb) edge[draw=black] (b_31);
        \path[-] (xb) edge[draw=black] (b_32);
        
        \path[-] (a_23) edge[draw=black, bend left=10] (a_31);
        \path[-] (a_33) edge[draw=black, bend right=10] (a_21);
        \path[-] (a_23) edge[draw=black] (a_33);
        
        \path[-] (xa) edge[draw=black] (a_22);
        \path[-] (xa) edge[draw=black] (a_32);
        \path[-] (xa) edge[draw=black] (a_23);
        \path[-] (xa) edge[draw=black] (a_33);

        \path[-] (xa) edge[draw=green] (xb);

    \end{tikzpicture}
    \end{minipage}
    \caption{The two types of connections between the sets $A$ and $B$ in the special case $r = 5$.}
    \label{fig:bridge_graph_h_odd}
\end{figure}

We begin by showing that we find three pairs of bridges of the first type, using an even number of vertices of both classes (\ref{stepd:bridge} of Section \ref{sec:o_extremal1}).

\begin{claim}\label{lemma_matching_edges_odd}
Suppose $\delta(G) \geq \tfrac{n+r-2}{2}$ and $\abs{A} \leq \abs{B}$. Furthermore, let $n$ be large enough. There is a matching $(x_{a_1}x_{b_1}, \ldots, x_{a_{r+1}}x_{b_{r+1}})$  such that $\abs{N(x_{a_i}) \cap A} \geq \tfrac n5$ and $\abs{N(x_{b_j}) \cap B} \geq \tfrac n5$ for all $i, j \leq r+1$.
\end{claim}
\begin{claimproof}
As in the proof of Claim~\ref{lemma_matching_edges}, it suffices to find $r+1$ edges from $A$ to $V \setminus A$. 

First, suppose that $n$ is odd. If $\abs{A} \leq \tfrac{n-r-2}{2}$, the minimum degree of $ \tfrac{n+r-2}{2}$ guarantees that each vertex of $A$ needs to find at least $\tfrac{n+r-2}{2} - (\tfrac{n-r-2}{2} - 1) = r+1$ neighbors outside of $A$, hence the assertion follows.
Suppose $\abs{A} = \tfrac{n-r}{2} + i$ with $i = 1, \ldots,\floor{\tfrac r2}$. In this case, $\abs{V \setminus A} = \tfrac{n+r}{2} - i$. Each vertex of $A$ finds at least $r-i$ neighbors outside of $A$. Suppose that $N(A) \setminus A$ has size at most $r$ (thus, all edges from $A$ into the rest of the graph belong to at most $r$ vertices). Now pick a different vertex in the complement (which exists, as $\tfrac{n+r}{2} - i \gg r$). This vertex requires $\tfrac{n+r-2}{2} - (\tfrac{n+r}{2} - i - 1) = i$ neighbors in $A$, which is a contradiction.
If finally $\abs{A} = \tfrac{n-r}{2}$, each vertex of $A$ has at least $r$ neighbors in the complement of $A$. If all vertices of $A$ share those $r$ vertices (hence, we only find $r$ matching edges), those vertices are connected to all vertices in $A$, hence can be moved to $A$ by only increasing $\alpha$-extremity, thus the assertion follows from the previous case.

Now, if $n$ is even, because the minimum degree needs to be an integer, it is at least $\tfrac{n+1+r-2}{2}$, hence upon removal of one vertex, we are left with a graph on $n' = n-1$ vertices and minimum degree at least $\tfrac{n+1+r-2}{2} -1= \tfrac{n'+r-2}{2}$ (and $n'$ being odd). Hence the assertion follows from the previous discussion.
\end{claimproof}

Similarly as in the even case, Claim~\ref{lemma_matching_edges_odd} gives us $\tfrac{r+1}{2}$ pairs of bridge-edges as in Figure \ref{fig:bridge_graph_h_odd} (the green part). Clearly, the rest of the bridge graph can be created completely analogously to Claim~\ref{bridge_graph_existence}.

By the same token we get the following claim immediately.
\begin{claim}\label{cor_matching_edges_odd}
Suppose $\delta(G) \geq \tfrac{n+r-2}{2}$ and $\abs{A} \leq \abs{B}$. Furthermore, let $n$ be large enough. There is a matching $(x_{a_1}x_{b_1}, \ldots, x_{a_{r-1}}x_{b_{r-1}})$  such that $\abs{N(x_{a_i}) \cap A} \geq  \tfrac n5$ and $\abs{N(x_{b_j}) \cap B} \geq \tfrac n5$ for all $i, j \leq r-1$. Furthermore, there are different vertices $x_a \in A, x_b \in B$ such that $x_ax_b \in E(G)$, $\abs{N(x_a) \cap A} \geq r-1$ and $\abs{ N(x_b) \cap B} \geq r-1$.
\end{claim}
\begin{claimproof}
This follows directly from the proof of Claim~\ref{lemma_matching_edges_odd}.
\end{claimproof}

Next, we re-define the gluing operation GE to GO as follows.
\begin{claim}[Gluing operation GO]
Given two disjoint sets $D_1, D_2 \subset A'_0$ of sizes exactly $\tfrac{r+1}{2}$, we find two disjoint sets $D'_1, D'_2 \subset A'_0 \setminus (D_1 \cup D_2)$ of size $\tfrac{r+1}{2}$ such that
\[G[D_1, D'_1] \equiv G[D'_1, D'_2] \equiv G[D'_2, D_2] \equiv K_{(r+1)/2, (r+1)/2} \, .\]
\end{claim}
\begin{claimproof}
This follows directly from the fact that each vertex in $A'_0$ is connected to at least $(\tfrac 12 - 4 \alpha)n$ vertices in $A'_0$ and GO uses only finitely many vertices of the neighborhoods.
\end{claimproof}
We stress at this point that GO can be applied to blocks whose end-vertices $D_1$ have currently degree $\tfrac{r-1}{2}$ (then we chose all edges from the connecting graphs $K_{(r+1)/2, (r+1)/2}$) or degree $\tfrac{r+1}{2}$ (then we chose the complete bipartite graph between $D'_1$ and $D'_2$ and at the other connections, we remove one matching of size $\tfrac{r+1}{2}$.
Furthermore observe, that gluing consumes $r+1$ vertices from the underlying set.

We proceed as follows. If $\abs{C} > 0$ or $\abs{C} = 0$ and $\abs{A'_0}, \abs{B'_0} \equiv 0 \pmod 2$, we create $\tfrac{r+1}{2}$ pairs of bridge vertices by Claim~\ref{lemma_matching_edges_odd}. Otherwise, we create $\tfrac{r-1}{2}$ pairs of bridge vertices and one additional bridge by Claim~\ref{cor_matching_edges_odd}. In both cases, we glue the bridges in $A$ and $B$ respectively together using mutually disjoint vertex sets $D^A_1, \ldots,D^A_{(r-1)/2}$ and $D^B_1, \ldots,D^B_{(r-1)/2}$ constructing $P_A, P_B$ and, similarly as in the even case, we set $A' = A'_0 \setminus V(P_A), B' = B'_0 \setminus V(P_B)$. 
Clearly, the parity of $A'$ and $B'$ are both even.

Next, we define absorbing structures for the left-over vertices and for absorbing vertices in order to guarantee divisibility (\ref{stepd:absorb} of Section \ref{sec:o_extremal1}). They need to be defined slightly differently as in the even case (Figure \ref{fig:absorber_odd}).

\begin{definition}
    Let $D \in \cbc{A', B'}$ and $u$ a vertex such that $\abs{ N(u) \cap D } \geq \tfrac n6$. Define $\xi_r(u)$ as follows.
    \begin{itemize}
         \item Select $D_1 = \cbc{d_1, d_2, \ldots, d_{(r+1)/2}}, D_2 = \cbc{d'_1, \ldots, d'_{(r+1)/2}} \subset  N(u) \cap D$, hence $r$ pairwise disjoint vertices.
         \item Select $D' = \cbc{ u'_1, \ldots, u'_{(r-1)/2}} \subset N(D_1) \cap N(D_2) \cap D \setminus (D_1 \cup D_2 \cup \cbc{u}) $.
         \item Select $E_0 = \cbc{ e_1, \ldots, e_{(r+1)/2} }$ in the joint $D$-neighborhood of $D_1 \setminus (D_2 \cup D' \cup \cbc{u})$.  
    \end{itemize}
    Define $\xi_r(u)$ as the graph containing $E_0, D_1, D_2, D'$ and $u$ as well as all the edges from $D_1$ to $D' \cup \cbc{u}$. Furthermore, take all edges from $D_2$ to $D' \cup \cbc{u}$ removing one matching of size $\tfrac{r+1}{2}$ and the edges from $E_0$ to $D_1$ removing a matching as well.
    Furthermore,  for two adjacent vertices $u_1, u_2 \in D$ define $\xi'_r(u_1, u_2)$ via
    \begin{itemize}
         \item Select $F_1 = \cbc{f_1, f_2, \ldots, f_{(r+1)/2}}, F' = \cbc{f'_1, \ldots, f'_{(r+1)/2}} \subset N(u_1) \cap N(u_2) \cap D$, hence $r+1$ pairwise disjoint vertices in the joint neighborhood of $u_1$ and $u_2$.
         \item Select the following edges
         \begin{itemize}
             \item $f_1u_1, \ldots, f_{(r-1)/2}u_1$,
             \item $f'_1u_1, \ldots, f'_{(r-1)/2}u_1$,
             \item $f_2u_2, \ldots, f_{(r+1)/2}u_2$,
             \item $f'_2u_2, \ldots, f'_{(r+1)/2}u_2$ and
             \item $f_1f'_1$ as well as $f_{(r+1)/2}f'_{(r+1)/2}$.
         \end{itemize}
         \item Finally, draw $(r-2)$ half-edges at each $f_i, f'_i$ and match them such that a simple graph is induced by the matching.
     \end{itemize}
\end{definition}
As $u_1, u_2 \in D$, hence the neighborhood of $u_1$, $u_2$ contains only vertices that, themselves, have $(\tfrac 12 - 30 r \alpha)n$ vertices from $D$, the joint neighborhood of those two vertices has size at least $(\tfrac 12 - 60 r \alpha )n$, $\xi'(u_1, u_2)$ is well defined. Observe that absorbing a vertex $u \not \in D$ consumes $2r + 1$ vertices from $D$ while absorbing $u_1, u_2 \in D$ consumes $r+3$ vertices (including $u_1, u_2$) in $D$.

As in the even case, we find a family of disjoint structures to absorb at least $100 \alpha n$ different vertices.

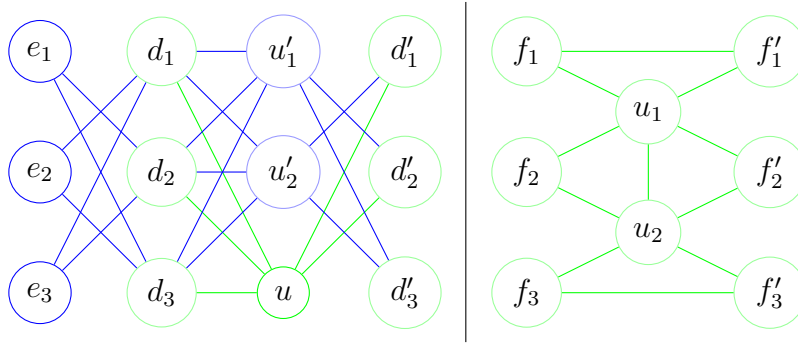
\begin{figure}
    \centering
    \begin{tikzpicture}[scale=0.8]
    
    \node[circle, draw=blue] (e1) at (0, 3) {$e_1$};
     \node[circle, draw=blue] (e2) at (0, 1) {$e_2$};
     \node[circle, draw=blue] (e3) at (0, -1) {$e_3$};

     \node[circle, draw=green!40] (d1) at (2, 3) {$d_1$};
     \node[circle, draw=green!40] (d2) at (2, 1) {$d_2$};
     \node[circle, draw=green!40] (d3) at (2, -1) {$d_3$};

     \node[circle, draw=green!40] (dp1) at (6, 3) {$d'_{1}$};
     \node[circle, draw=green!40] (dp2) at (6, 1) {$d'_{2}$};
     \node[circle, draw=green!40] (dp3) at (6, -1) {$d'_{3}$};
     
     \node[circle, draw=blue!40] (up1) at (4, 3) {$u'_{1}$};
     \node[circle, draw=blue!40] (up2) at (4, 1) {$u'_{2}$};
     \node[circle, draw=green] (u) at (4, -1) {$u$};

     \path[-] (u) edge[draw=green] (d1);
     \path[-] (u) edge[draw=green] (d2);
     \path[-] (u) edge[draw=green] (d3);
     \path[-] (u) edge[draw=green] (dp1);
     \path[-] (u) edge[draw=green] (dp2);
     
     \path[-] (d1) edge[draw=blue] (up1);
     \path[-] (d2) edge[draw=blue] (up1);
     \path[-] (d1) edge[draw=blue] (up2);
     \path[-] (d2) edge[draw=blue] (up2);
     \path[-] (d3) edge[draw=blue] (up1);
     \path[-] (d3) edge[draw=blue] (up2);
     
     \path[-] (d2) edge[draw=blue] (e1);
     \path[-] (d3) edge[draw=blue] (e1);
     \path[-] (d1) edge[draw=blue] (e2);
     \path[-] (d3) edge[draw=blue] (e2);
     \path[-] (d1) edge[draw=blue] (e3);
     \path[-] (d2) edge[draw=blue] (e3);

     \path[-] (dp2) edge[draw=blue] (up1);
     \path[-] (dp3) edge[draw=blue] (up1);
     \path[-] (dp1) edge[draw=blue] (up2);
     \path[-] (dp3) edge[draw=blue] (up2);
     
     
     \node (h1) at (7, 4) {};
     \node (h2) at (7, -2) {};

     \node[circle, draw=green!40] (f1) at (8, 3) {$f_1$};
     \node[circle, draw=green!40] (f2) at (8, 1) {$f_2$};
     \node[circle, draw=green!40] (f3) at (8, -1) {$f_3$};
     
     \node[circle, draw=green!40] (u1) at (10, 2) {$u_1$};
     \node[circle, draw=green!40] (u2) at (10, 0) {$u_2$};
     
     \node[circle, draw=green!40] (fp1) at (12, 3) {$f'_1$};
     \node[circle, draw=green!40] (fp2) at (12, 1) {$f'_2$};
     \node[circle, draw=green!40] (fp3) at (12, -1) {$f'_3$};

     \path[-] (h1) edge[draw=black] (h2);
     
     \path[-] (fp1) edge[draw=green] (u1);
     \path[-] (fp2) edge[draw=green] (u2);
     \path[-] (u2) edge[draw=green] (u1);
     
     \path[-] (fp2) edge[draw=green] (u1);
     \path[-] (fp3) edge[draw=green] (u2);
     \path[-] (f1) edge[draw=green] (u1);
     \path[-] (f2) edge[draw=green] (u2);
     
     \path[-] (f2) edge[draw=green] (u1);
     \path[-] (f3) edge[draw=green] (u2);
     
     \path[-] (fp1) edge[draw=green] (f1);
     \path[-] (fp3) edge[draw=green] (f3);

    \end{tikzpicture}
    
    \caption{Absorbers $\xi_5(u)$ (left) and $\xi'_5(u_1, u_2)$ (right) with $r=5$, where $u, u_1, u_2$ are the green vertices, the green vertices are inside the $A'$-neighborhood of $u$ (or $u_1, u_2$ respectively) and the blue vertices are vertices chosen from $A'$.}
    \label{fig:absorber_odd}
\end{figure}

Subsequently, we absorb $C$ using independent copies of $\xi_r(\cdot)$ such that the parity of the remaining vertices in the almost cliques is even. Then, as above, we glue the absorbers together by GO and are left with sets $A''$ and $B''$ (hence, $A'$ without the glued structures) and the path-like subgraph $P'_A$ of size at most $25r \alpha n$, hence $\abs{A''} \geq (\tfrac 12 - 30 r \alpha n)$. As each vertex of $C$ has degree at least $\tfrac{\alpha}{100} n$ into $A''$ and $B''$, we can absorb one vertex such that in the end $A''$ and $B''$ contain an even number of vertices. Furthermore, each vertex in $A''$ is connected to at least $(\tfrac 12 - 40 r \alpha)n$ vertices in $A''$ and the same applies to $B''$. 

Again, as in the even case, we need to make sure that $\abs{A''} \equiv 0 \pmod{2r}$, as we want to close the cycle by blocks of subsequently fully connected layers of sizes $\tfrac{r+1}{2}$ where the second block misses one matching of size $r+1$. If the divisibility condition holds, set $A''' = A''$ and proceed. Otherwise, if there is $0 < i < 2r$ such that $\abs{A''} \equiv i \pmod{2r}$, $i$ has to be even as the parity of $A''$ guarantees. Select $\tfrac i2$ pairs vertices $(a_{11}, a_{12}), \ldots,(a_{i/2, 1} a_{i/2, 2}) \in A'' \times A''$ and absorb them using disjoint instances $\xi'_r(a_{11}, a_{12}), \ldots, \xi'_r(a_{i/2, 1} a_{i/2, 2})$ with $D = A''$. As each absorber consumes $r + 3$ vertices, the divisibility condition now holds. Finally, as in the even case, glue the absorbed parts together with GO which does not change the divisibility by $r+1$. Thus, we are left with a set $A'''$ which consists of the vertices of $A''$ without the absorbed vertices and the gluing structures. Analogously, the same applies for $B'''$. Now, as in the even case, the result directly follows from Lemma~\ref{lem:blowup} and the following claim (\ref{stepd:cover} of Section \ref{sec:o_extremal1}).

\begin{claim}
The constructed subgraph is $r$-connected and $r$-regular.
\end{claim}
\begin{claimproof}
As in the even case, $r$-regularity as well as $r$-connectivity on the $\tfrac{r+1}{2}$-blow-up of the path part is obvious.
The first type of bridge (build with Claim~\ref{lemma_matching_edges_odd}, see Figure~\ref{fig:bridge_graph_h_odd} on the left) does not harm connectivity as before. In the second type of bridge (build with Claim~\ref{cor_matching_edges_odd}, see Figure~\ref{fig:bridge_graph_h_odd} on the right) only the special vertices $x_a$ and $x_b$ need our attention. But as they are of degree $r$ and connected to $(\tfrac{r-1}{2}$ vertices on both sides of the $K_{(r+1)/2), (r+1)/2}$, isolating a part of the graph is not possible either. The absorbing structures clearly sustain the connectivity property.
\end{claimproof}
This finishes the proof of the first extremal case.
\end{proof}

\section{Extremal Case II}
\label{sec:extremal2}
In this section we deal with the second extremal case and follow~\ref{stepe:find}--\ref{stepe:cover} as outlined in Section~\ref{sec:o_extremal2}.
We start by proving the auxiliary lemma for finding stars.

\begin{proof}[Proof of Lemma~\ref{lem:stars}]
Let $\alpha = \tfrac{1}{32s^2(s+1)}$.
Assume we have already found $0 \le t < 2m$ copies of $K_{1,s}$ and let $V'$ be the remaining vertices.
Then, by the maximum degree condition in $G$,
\[e(G[V']) \ge \tfrac 12 n (m+s-1) - t (s+1) 4 s \alpha n \,.\]
If $m \ge s+1$ this is at least $\tfrac 14 n (m+s-1) \ge \tfrac 12 ns$ and gives a vertex of degree at least $s$ in $G[V']$.
On the other hand, if $m \le s$ the above is at least $(\tfrac 12 s - \tfrac 14) n > \tfrac{s-1}{2}n$ and again this gives a vertex of degree at least $s$ in $G[V']$.
\end{proof}

\begin{proof}[Proof of Extremal Case II]
Let $r \ge 2$ and $s = \lceil \tfrac{r}{2} \rceil \ge 1$.
Let $\eps>0$ be given by Lemma~\ref{lem:blowup} on input $\tfrac 12$, $\tfrac 12$, and $r$.
We obtain $\alpha>0$ from Lemma~\ref{lem:stars} and additionally assume that $40 s^2 \alpha < \eps$.
Then let $G$ be an $n$-vertex graph with minimum degree $\delta(G) \ge \tfrac{n+r-2}{2}$ and $nr \equiv 0 \pmod 2$.
Further assume, that there is a partition of $V(G)$ into $A$ and $B$ of size $|A|+m=\tfrac 12 n=|B|-m$, where $0 \le m \le \alpha n$ such that between these sets we have minimum degree $\alpha n$ and all but at most $\alpha n$ vertices from $A$ (or $B$) have degree $|B|-\alpha n$ (or $|A|-\alpha n$) into $B$ (or $A$).

\noindent \textbf{\ref{stepe:find}.}
Note $\delta(G[B]) \ge m+s-1$.
Let $B' \subseteq B$ be the vertices of degree at most $2 s \alpha n$ in $G[B]$ and let $m'=|B \setminus B'|$.
If $m'<m$, then $\delta(G[B']) \ge (m-m')+s-1$ and we apply Lemma~\ref{lem:stars} to find $2(m-m')$ copies of $K_{1,s}$.
By choice of $B'$ each vertex from these copies of $K_{1,s}$ has degree at least $|A|-2s\alpha n$ into $A$.
Let $W$ be the union of the vertices from these copies of $K_{1,s}$

Afterwards, for $i=1,\dots,\min\{m', m\}$ we can pick a vertex $x_i \in B \setminus (B' \cup \{x_1,\dots,x_{i-1} \})$ and neighbours $y_{i,1},\dots,y_{i,r}$ of $x_i$ from $A \setminus (W \cup \bigcup_{j=1}^{i-1} \{y_{j,1},\dots,y_{j,r} \}$ such that $y_{i,1},\dots,y_{i,r}$ have degree at least $|A|-\alpha n$  into $A$.
For some $0 \le m' \le m$ we have obtained $m'$ copies of  $K_{1,r}$ and $2(m-m')$ copies of $K_{1,s}$ such that all vertices in the copies of $K_{1,s}$ and the leaves in the copies of $K_{1,r}$ have degree at least $|A|-2s\alpha n$ into $A$.

\noindent \textbf{\ref{stepe:absorb}.}
We will now iteratively absorb these copies of $K_{1,s}$ and $K_{1,r}$ into an $r$-regular path-structure.
Let $W$ be the set of vertices in the union of these copies.
We start with $r$ vertices $\mathbf{v}=(v_1,\dots,v_s)$ from $A \setminus W$ that have degree at least $|A|-\alpha n$ into $A$ and $s$ vertices $\mathbf{u}_0=(u_{0,1},\dots,u_{0,s})$ from $B \setminus W$ that have degree at least $|B|-\alpha n$ into $B$ such that $v_ju_{0,j'}$ is an edge for $1 \le j, j' \le s$.
We add the vertices in $\mathbf{v}$ and $\mathbf{u}_0$ to $W$.
Now for some $i=0,\dots,m$ let $W$ be the vertices used for the structure and in the remaining copies of $K_{1,s}$ and $K_{1,s}$ with $|W| \le 6 s i + 2s$ and let $\mathbf{u}_i=(u_{i,1},\dots,u_{i,r})$ be one end of the structure in $B$ such that $u_{i,j}$ has at least $|A|-2s\alpha n$ neighbours in $A$ for $j=1,\dots,s$.

If there are two copies of $K_{1,s}$ left then by alternating between $A$ and $B$ we find the structure shown in Figure~\ref{fig:absorbK1s} where the vertices on the left are $\mathbf{u}_i$ and the vertices on the right we denote by $\mathbf{u}_{i+1}=(u_{i+1,1},\dots,u_{i+1,s})$.
Note that $u_{i+1,1},\dots,u_{i+1,s}$ are exactly the leaves of one of the $K_{1,s}$ and, therefore, have at least degree $|A|-2s\alpha n$ into $A$.

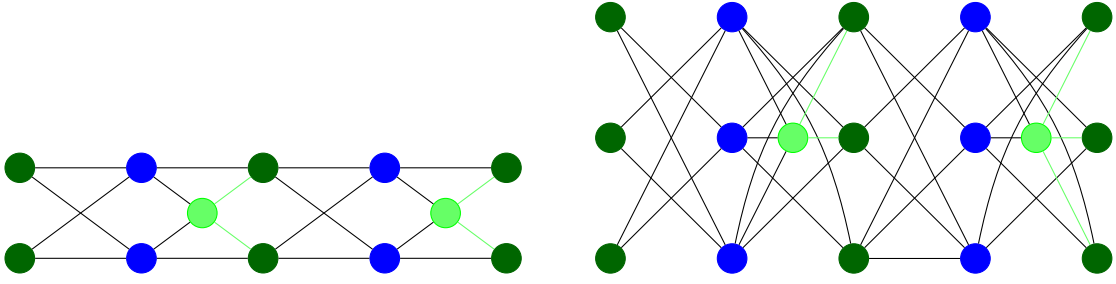
\begin{figure}
    \centering
    \begin{tikzpicture}[scale=0.8]
     \node[circle, draw=black!60!green, fill=black!60!green] (b11) at (-6, 0) {};
     \node[circle, draw=black!60!green, fill=black!60!green] (b12) at (-6, -1.5) {};

     \node[circle, draw=blue, fill=blue] (a11) at (-4, 0) {};
     \node[circle, draw=blue, fill=blue] (a12) at (-4, -1.5) {};
     
     \node[circle, draw=green, fill=green!60] (b1) at (-3, -0.75) {};
     
     \node[circle, draw=black!60!green, fill=black!60!green] (b21) at (-2, 0) {};
     \node[circle, draw=black!60!green, fill=black!60!green] (b22) at (-2, -1.5) {};

     \node[circle, draw=blue, fill=blue] (a21) at (-0, 0) {};
     \node[circle, draw=blue, fill=blue] (a22) at (-0, -1.5) {};
     
     \node[circle, draw=green, fill=green!60] (b2) at (1, -0.75) {};
     
     \node[circle, draw=black!60!green, fill=black!60!green] (b31) at (2, 0) {};
     \node[circle, draw=black!60!green, fill=black!60!green] (b32) at (2, -1.5) {};

     \path[-] (b11) edge[draw=black] (a11);
     \path[-] (b11) edge[draw=black] (a12);
     \path[-] (b12) edge[draw=black] (a11);
     \path[-] (b12) edge[draw=black] (a12);
    
     \path[-] (a11) edge[draw=black] (b21);
     \path[-] (a11) edge[draw=black] (b1);
     \path[-] (a12) edge[draw=black] (b22);
     \path[-] (a12) edge[draw=black] (b1);
     
     \path[-] (b1) edge[draw=green!60] (b21);
     \path[-] (b1) edge[draw=green!60] (b22);
     
     \path[-] (b21) edge[draw=black] (a21);
     \path[-] (b21) edge[draw=black] (a22);
     \path[-] (b22) edge[draw=black] (a21);
     \path[-] (b22) edge[draw=black] (a22);
    
     \path[-] (a21) edge[draw=black] (b31);
     \path[-] (a21) edge[draw=black] (b2);
     \path[-] (a22) edge[draw=black] (b32);
     \path[-] (a22) edge[draw=black] (b2);
     
     \path[-] (b2) edge[draw=green!60] (b31);
     \path[-] (b2) edge[draw=green!60] (b32);

    
    \end{tikzpicture}
    \qquad 
    \begin{tikzpicture}[scale=0.8]
    \node[circle, draw=black!60!green, fill=black!60!green] (b11) at (-6, 1) {};
    \node[circle, draw=black!60!green, fill=black!60!green] (b12) at (-6, -1) {};
    \node[circle, draw=black!60!green, fill=black!60!green] (b13) at (-6, -3) {};

     \node[circle, draw=blue, fill=blue] (a11) at (-4, 1) {};
     \node[circle, draw=blue, fill=blue] (a12) at (-4, -1) {};
     \node[circle, draw=blue, fill=blue] (a13) at (-4, -3) {};
     
     \node[circle, draw=green, fill=green!60] (b1) at (-3, -1) {};
     
     \node[circle, draw=black!60!green, fill=black!60!green] (b21) at (-2, 1) {};
     \node[circle, draw=black!60!green, fill=black!60!green] (b22) at (-2, -1) {};
     \node[circle, draw=black!60!green, fill=black!60!green] (b23) at (-2, -3) {};

     \node[circle, draw=blue, fill=blue] (a21) at (-0, 1) {};
     \node[circle, draw=blue, fill=blue] (a22) at (-0, -1) {};
     \node[circle, draw=blue, fill=blue] (a23) at (-0, -3) {};
     
     \node[circle, draw=green, fill=green!60] (b2) at (1, -1) {};
     
     \node[circle, draw=black!60!green, fill=black!60!green] (b31) at (2, 1) {};
     \node[circle, draw=black!60!green, fill=black!60!green] (b32) at (2, -1) {};
     \node[circle, draw=black!60!green, fill=black!60!green] (b33) at (2, -3) {};

     \path[-] (b11) edge[draw=black] (a12);
     \path[-] (b11) edge[draw=black] (a13);
     \path[-] (b12) edge[draw=black] (a11);
     \path[-] (b12) edge[draw=black] (a13);
     \path[-] (b13) edge[draw=black] (a11);
     \path[-] (b13) edge[draw=black] (a12);
    
     \path[-] (a11) edge[draw=black] (b22);
     \path[-] (a11) edge[draw=black] (b1);
     \path[-] (a11) edge[draw=black, bend left = 15] (b23);
     
     \path[-] (a12) edge[draw=black] (b21);
     \path[-] (a12) edge[draw=black] (b1);
     \path[-] (a12) edge[draw=black] (b23);
     
     \path[-] (a13) edge[draw=black, bend left = 15] (b21);
     \path[-] (a13) edge[draw=black] (b1);
     \path[-] (a13) edge[draw=black] (b22);
     
     \path[-] (b1) edge[draw=green!60] (b21);
     \path[-] (b1) edge[draw=green!60] (b22);
     
     \path[-] (b21) edge[draw=black] (a22);
     \path[-] (b21) edge[draw=black] (a23);
     \path[-] (b22) edge[draw=black] (a21);
     \path[-] (b22) edge[draw=black] (a23);
     \path[-] (b23) edge[draw=black] (a21);
     \path[-] (b23) edge[draw=black] (a22);
     \path[-] (b23) edge[draw=black] (a23);
    
     \path[-] (a21) edge[draw=black] (b32);
     \path[-] (a21) edge[draw=black] (b2);
     \path[-] (a21) edge[draw=black, bend left = 15] (b33);
    
     \path[-] (a22) edge[draw=black] (b31);
     \path[-] (a22) edge[draw=black] (b2);
     \path[-] (a22) edge[draw=black] (b33);
     
     \path[-] (a23) edge[draw=black, bend left = 15] (b31);
     \path[-] (a23) edge[draw=black] (b32);
    
     \path[-] (b2) edge[draw=green!60] (b31);
     \path[-] (b2) edge[draw=green!60] (b32);
     \path[-] (b2) edge[draw=green!60] (b33);
     \end{tikzpicture}
    
    \caption{Including two copies of $K_{1,s}$ into an $r$-regular path structure in the even ($r=4$) and odd case ($r=5$) with blue vertices in $A$ and green vertices in $B$.}
    \label{fig:absorbK1s}
\end{figure}

Otherwise, we pick a copy of $K_{1,r}$ and find the structure shown in Figure~\ref{fig:absorbK1r} (this is an $\tfrac r2$-blow-up of a path), where the vertices on the left are $\mathbf{u}_i$ and the vertices on the right we denote by $\mathbf{u}_{i+1}=(u_{i+1,1},\dots,u_{i+1,s})$.
We do not need any assumption on the degree of the centre vertex of the $K_{1,r}$ as we do not need any additional edges containing it.
Again $u_{i+1,1},\dots,u_{i+1,r}$ have at least degree $|A|- 2s\alpha n$ into $A$.
\begin{figure}
    \centering
    \begin{tikzpicture}[scale=0.8]
     \node[circle, draw=black!60!green, fill=black!60!green] (b11) at (-6, 0) {};
     \node[circle, draw=black!60!green, fill=black!60!green] (b12) at (-6, -1.5) {};

     \node[circle, draw=blue, fill=blue] (a11) at (-4, 0) {};
     \node[circle, draw=blue, fill=blue] (a12) at (-4, -1.5) {};
     
     \node[circle, draw=black!60!green, fill=black!60!green] (b21) at (-2, 0) {};
     \node[circle, draw=black!60!green, fill=black!60!green] (b22) at (-2, -1.5) {};

     \node[circle, draw=green!60, fill=green!60] (b1) at (-0, 0) {};
     \node[circle, draw=blue, fill=blue] (a22) at (-0, -1.5) {};

     \node[circle, draw=black!60!green, fill=black!60!green] (b31) at (2, 0) {};
     \node[circle, draw=black!60!green, fill=black!60!green] (b32) at (2, -1.5) {};

     \path[-] (b11) edge[draw=black] (a11);
     \path[-] (b11) edge[draw=black] (a12);
     \path[-] (b12) edge[draw=black] (a11);
     \path[-] (b12) edge[draw=black] (a12);
    
     \path[-] (a11) edge[draw=black] (b21);
     \path[-] (a11) edge[draw=black] (b22);
     \path[-] (a12) edge[draw=black] (b22);
     \path[-] (a12) edge[draw=black] (b21);

     \path[-] (b21) edge[draw=green!60] (b1);
     \path[-] (b21) edge[draw=black] (a22);
     \path[-] (b22) edge[draw=green!60] (b1);
     \path[-] (b22) edge[draw=black] (a22);
    
     \path[-] (b1) edge[draw=green!60] (b31);
     \path[-] (a22) edge[draw=black] (b31);
     \path[-] (b1) edge[draw=green!60] (b32);
     \path[-] (a22) edge[draw=black] (b32);
     
    \end{tikzpicture}
    \qquad 
    \begin{tikzpicture}[scale=0.8]
    \node[circle, draw=black!60!green, fill=black!60!green] (b11) at (-6, 1) {};
    \node[circle, draw=black!60!green, fill=black!60!green] (b12) at (-6, -1) {};
    \node[circle, draw=black!60!green, fill=black!60!green] (b13) at (-6, -3) {};

     \node[circle, draw=blue, fill=blue] (a11) at (-4, 1) {};
     \node[circle, draw=blue, fill=blue] (a12) at (-4, -1) {};
     \node[circle, draw=blue, fill=blue] (a13) at (-4, -3) {};

     \node[circle, draw=black!60!green, fill=black!60!green] (b21) at (-2, 1) {};
     \node[circle, draw=black!60!green, fill=black!60!green] (b22) at (-2, -1) {};
     \node[circle, draw=black!60!green, fill=black!60!green] (b23) at (-2, -3) {};

     \node[circle, draw=green!60, fill=green!60] (b1) at (-0, 1) {};
     \node[circle, draw=blue, fill=blue] (a22) at (-0, -1) {};
     \node[circle, draw=blue, fill=blue] (a23) at (-0, -3) {};

     \node[circle, draw=black!60!green, fill=black!60!green] (b31) at (2, 1) {};
     \node[circle, draw=black!60!green, fill=black!60!green] (b32) at (2, -1) {};
     \node[circle, draw=black!60!green, fill=black!60!green] (b33) at (2, -3) {};

     \path[-] (b11) edge[draw=black] (a12);
     \path[-] (b11) edge[draw=black] (a13);
     \path[-] (b12) edge[draw=black] (a11);
     \path[-] (b12) edge[draw=black] (a13);
     \path[-] (b13) edge[draw=black] (a11);
     \path[-] (b13) edge[draw=black] (a12);
    
     \path[-] (a11) edge[draw=black] (b22);
     \path[-] (a11) edge[draw=black] (b21);
     \path[-] (a11) edge[draw=black] (b23);
     
     \path[-] (a12) edge[draw=black] (b21);
     \path[-] (a12) edge[draw=black] (b22);
     \path[-] (a12) edge[draw=black] (b23);
     
     \path[-] (a13) edge[draw=black] (b21);
     \path[-] (a13) edge[draw=black] (b22);
     \path[-] (a13) edge[draw=black] (b22);
     
     \path[-] (b21) edge[draw=black] (a22);
     \path[-] (b21) edge[draw=black] (a23);
     \path[-] (b22) edge[draw=green!60] (b1);
     \path[-] (b22) edge[draw=black] (a23);
     \path[-] (b23) edge[draw=green!60] (b1);
     \path[-] (b23) edge[draw=black] (a22);
    
     \path[-] (b1) edge[draw=green!60] (b32);
     \path[-] (b1) edge[draw=green!60] (b31);
     \path[-] (b1) edge[draw=green!60] (b33);
    
     \path[-] (a22) edge[draw=black] (b31);
     \path[-] (a22) edge[draw=black] (b32);
     \path[-] (a22) edge[draw=black] (b33);
     
     \path[-] (a23) edge[draw=black] (b31);
     \path[-] (a23) edge[draw=black] (b32);
     \path[-] (a23) edge[draw=black] (b33);

     \end{tikzpicture}
    
    \caption{Including one copy of $K_{1,r}$ into an $r$-regular path structure in the even ($r=4$) and odd case ($r=5$) with blue vertices in $A$ and green vertices in $B$.}
    \label{fig:absorbK1r}
\end{figure}
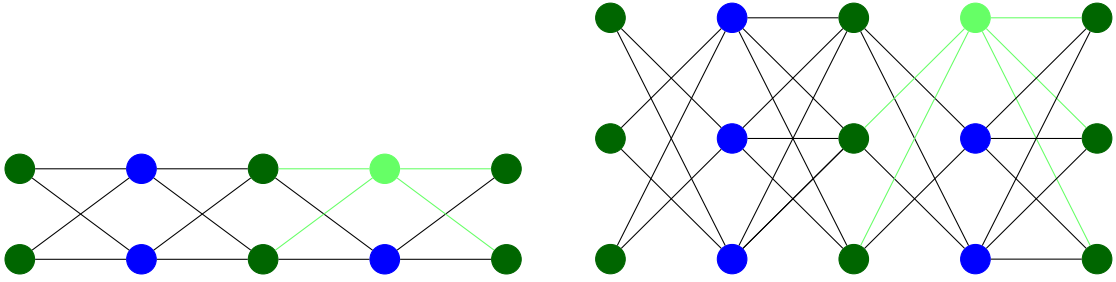

We add all new vertices in the structure to $W$. We can repeat this until all copies of $K_{1,s}$ and $K_{1,r}$ are covered, because in each step $W$ will increase by at most $6r$ vertices and all vertices have large enough neighbourhoods.
We let $\mathbf{u}'=(u_{m,1},\dots,u_{m,s})$ be the final end of this construction.
Now let $A_1=A \setminus W$ and $B_2=B \setminus W$ and note that $|A_1|=|B_1|$ because our construction uses $2m$ vertices more from $B$ than from $A$.

\noindent \textbf{\ref{stepe:absorb2}.}
For the next step let $A' \subseteq A_1$ be the vertices of degree at most $|A_1|-5 s \alpha n$ into $B_1$ and  $B' \subseteq B_1$ be the vertices of degree at most $|A_1|-5 s \alpha n$ into $A_1$.
Note that $|A'|,|B'| \le \alpha n$, because we removed at most $4 s \alpha n$ vertices from each of $A$ and $B$ to get $A_1$ and $B_1$.
By using an $\tfrac r2$-blow-up of a path (similar as in Figure~\ref{fig:absorbK1r}) we cover the vertices in $A'$ and $B'$ with our $r$-regular path structure covering in total at most $12 s \alpha n$ additional vertices.
For this we iteratively extend from $\mathbf{u}'$.
Let $W$ be the vertices used for covering $A'$ and $B'$ and let $\mathbf{u}=(u_1,\dots,u_s)$ be the last vertices of this construction in $B$ and note that we can assume that they have degree at least $|A|-20 s \alpha n$ into $A$.
We let $A_2=(A_1 \setminus W)$ and $B_2=(B_1 \setminus W)$ and note that $|A_2|=|B_2|$.

\noindent \textbf{\ref{stepe:cover}.}
We have that every vertex from $A_2$ (or $B_2$) has degree at least $|A_2|-20 s \alpha n$ into $B_2$ (or into $A_2$).
Then it is easy to see that $(A_2,B_2)$ is $(\eps,\tfrac 12)$-super-regular as $40 s \alpha < \eps$.
Moreover, the vertices in $\mathbf{v}$ have $|B_2| - 20s^2\alpha n \ge \tfrac 12 |B_2|$ common neighbours in $B_2$ and similarly for $\mathbf{u}$ with $A_2$.
We apply Lemma~\ref{lem:blowup} to cover the remaining vertices of $A_2$ and $B_2$ with the $\tfrac r2$-blow-up of a path such that the ends connect to $\mathbf{v}$ and $\mathbf{u}$.
This completes the construction of a spanning $r$-regular structure in $G$.
To see that it is also $r$-connected it suffices to note that we have the $\tfrac r2$-blow-up of a path except for some parts replaced by the graphs from Figure~\ref{fig:absorbK1s}, which do not harm this property.
\end{proof}

\section{Non-Extremal Case}
\label{sec:non-extremal}

In this section we deal with the case that $G$ is not $\alpha$-extremal.
Recall, that the assumption implies that for any two sets $A,B \subseteq V(G)$ of size $(\tfrac12 - \alpha)n \le |A|,|B| \le \tfrac n2$ we have $d(A,B) \ge \alpha$.
We will follow~\ref{step:regularise}--\ref{step:spanning} as outlined in Section~\ref{sec:o_non-extremal}.

\begin{proof}[Proof of Non-Extremal Case]
Given $r \ge 3$ and $0<\alpha<\tfrac{1}{32}$ we choose constants such that
\begin{align*}
    \varepsilon \ll \nu \ll d \ll \beta \ll \alpha\,,
\end{align*}
where, in particular,
\begin{align*}
    \eps \le \nu, \qquad \nu \le d^{s+1}, \qquad 10 s d \le \beta, \qquad 500 s \beta \le \alpha \,
\end{align*}
and $2 \eps$ is small enough for Lemma~\ref{lem:blowup} with input $\tfrac d2$, $\tfrac 14 d^s$, and $r$.
Let $M$ be given by Lemma~\ref{lem:regularity} on input $\eps$ and let $s=\lceil \tfrac r2 \rceil$.

\smallskip
\noindent \textbf{\ref{step:regularise}.}
Let $G$ be an $n$-vertex graph with minimum degree $\delta(G) \ge (\tfrac 12-\beta)n$.
From Lemma~\ref{lem:regularity} we get a partition of the vertex set $V(G)$ into $\ell+1 \le M$ clusters $V_0,\dots,V_\ell$ of size $L$ and a subgraph $G' \subseteq G$ such that~\ref{reg:size}--\ref{reg:reg} hold.
We denote by $R$ the graph on vertex set $[\ell]$ with edges $ij$ if and only if the pair $(V_i,V_j)$ is $(\eps,d)$-regular.
In $R$ we have minimum degree $\delta(R) \ge (\tfrac 12-\beta-2d)\ell$, because otherwise with~\ref{reg:ind} there would be a vertex with degree in $G'$ at most $(\tfrac 12-\beta-2d) \ell \cdot \tfrac{n}{\ell} + \eps n < (\tfrac 12-\beta) n - (d+\eps)n$ contradicting~\ref{reg:deg}.
Similarly, we can deduce that $R$ is not $\tfrac \alpha 2$-extremal.
Otherwise, there would be two sets of vertices $\cA$, $\cB$ in $R$ such that $(\tfrac 12-\tfrac \alpha2 )\ell \le |\cA|,|\cB| \le \tfrac 12 \ell$ and $d(\cA,\cB)< \tfrac \alpha2$.
Then $A=\bigcup_{i \in \cA} V_i$ and $B=\bigcup_{i \in \cB} V_i$ both have size at most $\tfrac \ell 2 \cdot \tfrac{n}{\ell} = \tfrac n2$ and at least $(\tfrac 12- \tfrac \alpha 2) \ell \cdot (1-\eps) \tfrac{n}{\ell} \ge (\tfrac 12-\alpha)n$ and we have
\[ d(A,B) = \frac{e(A,B)}{|A||B|} \le \frac{\tfrac \alpha 2 \cdot (\tfrac n \ell)^2 \cdot |\cA||\cB| + |A| (d+\eps) n}{|A| |B|} \le \frac{\tfrac \alpha 2}{(1-\eps)^2} + \frac{(d+\eps)}{\tfrac 12-\alpha} \le \alpha\,,\]
which contradicts our assumption that $G$ is not $\alpha$-extremal.
We will repeatedly use the following fact that holds as $R$ is not $\tfrac \alpha2$-extremal.
\begin{fact}
\label{fact:non-extremal}
    For any two sets $\cA, \cB \subseteq V(R)$ of size at least $(\tfrac 12 - \tfrac \alpha 2)\ell$ there is an edge $AB \in E(R)$ with $A \in \cA$ and $B \in \cB$.
\end{fact}
In the following we will abuse notation and also treat the clusters as vertices of $R$.

\smallskip
\noindent \textbf{\ref{step:matching}.}
Next let $M$ be a largest matching in $R$ and $M_0 \subseteq [\ell]$ be the clusters not covered by $M$.
Naturally, $M_0$ is an independent set in $R$ and if there are at least two vertices $u$ and $v$ in $M_0$ then no neighbour of $u$ is connected to a neighbour of $v$ by an edge of $M$.
Therefore, $2|M| \ge \deg_R(u)+\deg_R(v)$ and $|M | \ge (\tfrac 12-\beta-2d) \ell$.
We let $\ell'=|M|$ and denote the regular pairs corresponding to edges of $M$ by $(X_i,Y_i)$ for $i=1,\dots,\ell'$.

\smallskip
\noindent \textbf{\ref{step:connect}.}
For any clusters $Z$ and $W$ we call an $s$-tuple $\mathbf{z}=(z_1,\dots, z_s)$ from $Z$ \emph{well-connected into $W$} if the vertices $z_1,\dots,z_s$ have at least $\tfrac 12 d^s L$ common neighbours in $W$.
Now fix any $i \in [\ell']$.
We want to connect $Y_i$ and $X_{i+1}$ by the $\tfrac r2$-blow-up of a path.
For this we consider the neighbours $\cW$ and $\cZ$ in $R$ of $X_{i+1}$ and $Y_i$, respectively.
It follows from the minimum degree in $R$ that $\cW$ and $\cZ$ have size at least $(\tfrac 12 - \tfrac \alpha 2)\ell$.
By Fact~\ref{fact:non-extremal} there is $W \in \cW$ and $Z \in \cZ$ such that $WZ$ is an edge in $R$.

All but at most $2 s \eps L^s$ $s$-tuples $\mathbf{x}=(x_1,\dots, x_s)$ from $X_{i+1}$ are well connected into $W$ and $Y_{i+1}$.
The same holds for tuples from $Y_i$, $W$, and $Z$ with respect to the neighbouring clusters.
We fix tuples $\mathbf{x}=(x_1,\dots, x_s)$, $\mathbf{w}=(w_1,\dots, w_s)$, $\mathbf{z}=(z_1,\dots, z_s)$, and $\mathbf{y}=(y_1,\dots, y_s)$ from $X$, $W$, $Z$, and $Y$, respectively, such that $\mathbf{x}\mathbf{w}\mathbf{z}\mathbf{y}$ gives the $\tfrac r2$-blow up of a path on $4$ vertices.
We denote this path by $P_i$ and remove any internal vertices (those in $\mathbf{w}$ and $\mathbf{z}$) from the clusters.
We can repeat this for all $i$, because we need only $4 s \ell'$ vertices in total.

\smallskip
\noindent \textbf{\ref{step:superreg}.}
To make the matching edges super-regular, we let $i \in [\ell']$ and apply Lemma~\ref{lem:superreg} to the pair $(Y_i,X_i)$.
After removing a few additional vertices we arrive at sets $Y_i$ and $X_i$ such that $|Y_i|=|X_i| = L' \ge (1-2\eps)L$, where $L' \equiv 0 \pmod{s}$, and the pair $(Y_i,X_i)$ is $(2 \eps, d-2\eps)$-super-regular.
Note that we can do this such that the end-tuple $\mathbf{x}$ of $P_{i-1}$ is contained in $X_i$ and the end-tuple $\mathbf{y}$ of $P_{i}$ is contained in $Y_i$.
We add the vertices removed during this procedure and also the vertices that belong to clusters of $M_0$ to $V_0$ and note that $|V_0| \le \eps n + \ell' 4 \eps L + \beta n + 2dn \le 2 \beta n$.

\smallskip
\noindent \textbf{\ref{step:absorb}. Setup.}
We want to absorb $V_0$ by extending the paths $P_i$.
After each extension we need to maintain the location of the end-tuples and also ensure that they are well-connected.
During the procedure we will have to deal with sets of already covered vertices.
For this let $W_0$ be a set of size at most $18 s^2 \nu n$ and let $W$ be a set of size at most $20 s \beta n$.
These will be the sets of vertices that we already used.
There are at most $18 s^2 \nu n / (\tfrac 14 d^s \tfrac n\ell) = 72 s^2 \nu d^{-s}  \ell \le 8 s \beta \ell$ clusters that intersect $W_0$ in at least $\tfrac 14 d^s \tfrac n\ell$ vertices and at most $20 s \beta n / (\tfrac 12 \cdot \tfrac n\ell) \le 40 s \beta \ell$ clusters that intersect $W$ in at least $\tfrac 12 \cdot \tfrac n\ell$ vertices.
We denote by $H$ the set of all clusters that do not have this property.

Now consider a vertex $v \in V_0$.
There are at most $(\tfrac 12 + 50 s \beta) \ell$ clusters that intersect $N_G(v) \setminus (W \cup W_0)$ in less than $d \tfrac n \ell$ vertices.
Therefore, there are at least $(\tfrac 12 - 100 s \beta) \ell$ clusters in $H$ that intersect $N_G(v) \setminus (W \cup W_0)$ in at least $d \tfrac n \ell$ vertices.
We denote this set of clusters by $H(v)$.
Similarly, let $H_M(v)$ be the clusters form $H$, which share an edge of $M$ with another cluster from $H(v)$ and note that we have the same lower bound as $M$ is a matching.
Summing up we have $|H| \ge (1-50 s \beta)\ell$ and $|H(v)|,|H_M(v)| \ge (\tfrac 12 - 100 s \beta) \ell$ for all $v \in V_0$.
Note that $H(v)$ and $H_M(v)$ are large enough for Fact~\ref{fact:non-extremal}.

\smallskip
\textbf{Covering $2s$ vertices.}
We let $W=\emptyset$ and $W_0$ be all internal vertices (not in end-tuples) of the paths $P_1,\dots,P_\ell$.
We pick any $2s$ vertices $v_1,\dots,v_{2s}$ from $V_0 \setminus (W \cup W_0)$ and our goal is to embed them, such that in any pair $(X_i,Y_i)$ we still have $|X_i|=|Y_i| \equiv 0 \pmod s$.
Let $i_1$ be such that $X_{i_1} \in H(v_1)$ and $Y_{i_1} \in H$.
As the end-tuple of $P_{i_1}$ in $X_{i_1}$ is well-connected into $Y_{i_1}$ and $(N_G(v)\cap X_{i_1}) \setminus (W \cup W_0)$ is of size at least $d \tfrac n \ell$, we can greedily pick tuples $\mathbf{x}_1$, $\mathbf{x}_2$ in $X_{i_1}$ and $\mathbf{y}_1$, $\mathbf{y}_2$ in $Y_{i_1}$ with the exception that $\mathbf{y}_1$ contains $v_1$ and such that $\mathbf{x}_1\mathbf{y}_1\mathbf{x}_2\mathbf{y}_2P_{i_1}$ gives the $\tfrac r2$-blow-up of a path and $\mathbf{x_1}$ is well-connected into $Y_{i_1}$.
Now we remove the internal vertices (those in $\mathbf{x_2}$, $\mathbf{y_1}$, $\mathbf{y_2}$ and the end-tuple $\mathbf{x}$ of $P_{i_1}$) from the clusters and add them to $W_0$ (this adds $4s$ vertices; see Figure~\ref{fig:Step51}).
We note that $|Y_{i_1}|-1=|X_{i_1}|\equiv 0 \pmod s$ and let $P_{i_1}$ be the longer path.

\begin{figure}
    \centering
    \begin{tikzpicture}[scale=0.8, every node/.style={inner sep=0.2,outer sep=0.2}]

    \node[circle, draw=red, fill = red, minimum width = 6, font=\tiny] (v1) at (0,0) {};
    \node[] at (0.5, 0) {$v_1$};
    \node[circle, draw=orange, fill=orange, font=\tiny, minimum width = 6] (x11) at (-2.7,1) {};
    \node[circle, draw=orange, fill=orange, font=\tiny, minimum width = 6] (x12) at (-2.7,1.5) {};
    \node[] at (-3.2, 1.25) {$\mathbf{x_1}$};
    
    \node[circle, draw=black, fill=black, minimum width = 6, font=\tiny] (x21) at (-2.25,2.25) {};
    \node[circle, draw=black,fill=black, minimum width = 6,  font=\tiny] (x22) at (-2.25,2.75) {};
    \node[] at (-2.7, 2.5) {$\mathbf{x_2}$};
    
    \node[circle, draw=black, fill=black, minimum width = 6, font=\tiny] (y1) at (1,1) {};
    \node[] at (1.5, 1) {$y_1$};
    \node[circle, draw=black,fill=black, minimum width = 6,  font=\tiny] (y21) at (1,2.5) {};
    \node[circle, draw=black,fill=black, minimum width = 6,  font=\tiny] (y22) at (1,3) {};
    \node[] at (1.5, 2.75) {$\mathbf{y_2}$};

    \node[circle, draw=orange, fill=orange, font=\tiny, minimum width = 6] (x1) at (-2.7,3.5) {};
    \node[circle, draw=orange, fill=orange, font=\tiny, minimum width = 6] (x2) at (-2.7,4) {};
    \node[] at (-3.2, 3.75) {$\mathbf{x}$};
    
    \node[ellipse, draw = black, minimum width = 60, minimum height = 100, label = above:{$X_{i_1}$}] (Xi) at (-2.75,2.375) {};
    \node[ellipse, draw = black, minimum width = 60, minimum height = 100, label = above:{$Y_{i_1}$}] (Yi) at (1.25,2.375) {};
    
    \path[-] (Xi) edge[draw=blue!20, line width=20pt] (Yi);
    
    \path[-] (v1) edge[draw=red] (x11);
    \path[-] (v1) edge[draw=red] (x21);
    \path[-] (v1) edge[draw=red] (x12);
    \path[-] (v1) edge[draw=red] (x22);
    
    \path[-] (y1) edge[draw=black] (x11);
    \path[-] (y1) edge[draw=black] (x21);
    \path[-] (y1) edge[draw=black] (x12);
    \path[-] (y1) edge[draw=black] (x22);
    
    \path[-] (y21) edge[draw=black] (x11);
    \path[-] (y22) edge[draw=black] (x21);
    \path[-] (y22) edge[draw=black] (x12);
    \path[-] (y21) edge[draw=black] (x22);
    
    \path[-] (y21) edge[draw=black] (x1);
    \path[-] (y22) edge[draw=black] (x2);
    \path[-] (y21) edge[draw=black] (x2);
    \path[-] (y22) edge[draw=black] (x1);

     \end{tikzpicture}
    
    \caption{Absorbing vertex $v_1$ in the case $r=2s=4$ if $\mathbf{x}$ is the current end of path $P_i$ and $\mathbf{x_1}$ is the new end.}
    \label{fig:Step51}
\end{figure}
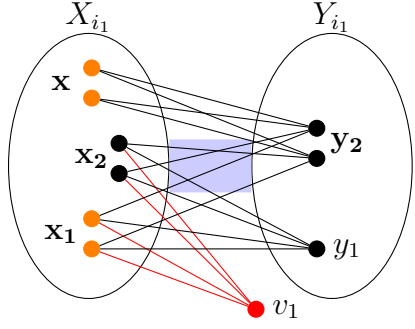

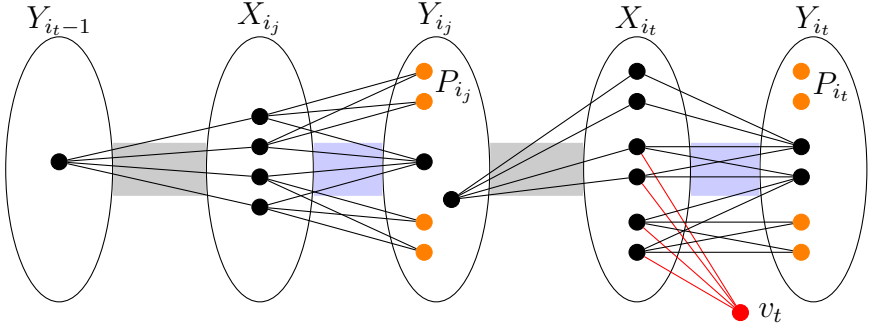
\begin{figure}
    \centering
    \begin{tikzpicture}[scale=0.8, every node/.style={inner sep=0.2,outer sep=0.2}]
    
    \node[circle, draw=black,fill=black, minimum width = 6,  font=\tiny] (yi) at (-11.2, 2.5) {};
    
    \node[circle, draw=black,fill=black, minimum width = 6,  font=\tiny] (xj1) at (-7.9, 1.75) {};
    \node[circle, draw=black,fill=black, minimum width = 6,  font=\tiny] (xj2) at (-7.9, 2.25) {};
    \node[circle, draw=black,fill=black, minimum width = 6,  font=\tiny] (xj3) at (-7.9, 2.75) {};
    \node[circle, draw=black,fill=black, minimum width = 6,  font=\tiny] (xj4) at (-7.9, 3.25) {};

    \node[circle, draw=black,fill=black, minimum width = 6,  font=\tiny] (yj) at (-4.75,1.875) {};
    
    \node[circle, draw=orange,fill=orange, minimum width = 6,  font=\tiny] (yj1) at (-5.2, 1) {};
    \node[circle, draw=orange,fill=orange, minimum width = 6,  font=\tiny] (yj2) at (-5.2, 1.5) {};
    \node[circle, draw=black,fill=black, minimum width = 6,  font=\tiny] (yj3) at (-5.2, 2.5) {};
    \node[circle, draw=orange,fill=orange, minimum width = 6,  font=\tiny] (yj4) at (-5.2, 3.5) {};
    \node[circle, draw=orange,fill=orange, minimum width = 6,  font=\tiny] (yj5) at (-5.2, 4) {};
    \node[] at (-4.7, 3.75) {$P_{i_j}$};

    \node[circle, draw=red, fill = red, minimum width = 6, font=\tiny] (v1) at (0,0) {};
    \node[] at (0.5, 0) {$v_t$};
    \node[circle, draw=black, fill=black, font=\tiny, minimum width = 6] (x11) at (-1.7,1) {};
    \node[circle, draw=black, fill=black, font=\tiny, minimum width = 6] (x12) at (-1.7,1.5) {};

    \node[circle, draw=black, fill=black, minimum width = 6, font=\tiny] (x21) at (-1.7,2.25) {};
    \node[circle, draw=black,fill=black, minimum width = 6,  font=\tiny] (x22) at (-1.7,2.75) {};
    
    \node[circle, draw=black, fill=black, minimum width = 6, font=\tiny] (x31) at (-1.7,3.5) {};
    \node[circle, draw=black,fill=black, minimum width = 6,  font=\tiny] (x32) at (-1.7,4) {};

    \node[circle, draw=orange, fill=orange, minimum width = 6, font=\tiny] (y11) at (1,1) {};
    \node[circle, draw=orange, fill=orange, minimum width = 6, font=\tiny] (y12) at (1,1.5) {};

    \node[circle, draw=black,fill=black, minimum width = 6,  font=\tiny] (y21) at (1,2.25) {};
    \node[circle, draw=black,fill=black, minimum width = 6,  font=\tiny] (y22) at (1,2.75) {};
    
    \node[circle, draw=orange,fill=orange, minimum width = 6,  font=\tiny] (y31) at (1,3.5) {};
    \node[circle, draw=orange,fill=orange, minimum width = 6,  font=\tiny] (y32) at (1,4) {};
    
    \node[] at (1.5, 3.75) {$P_{i_t}$};

    \node[ellipse, draw = black, minimum width = 40, minimum height = 100, label = above:{$Y_{i_{t}-1}$}] (Yi) at (-11.2,2.375) {};
    
    \node[ellipse, draw = black, minimum width = 40, minimum height = 100, label = above:{$Y_{i_j}$}] (Yj) at (-5,2.375) {};
    \node[ellipse, draw = black, minimum width = 40, minimum height = 100, label = above:{$X_{i_j}$}] (Xj) at (-7.9,2.375) {};
    
    \node[ellipse, draw = black, minimum width = 40, minimum height = 100, label = above:{$X_{i_t}$}] (Xt) at (-1.7,2.375) {};
    \node[ellipse, draw = black, minimum width = 40, minimum height = 100, label = above:{$Y_{i_t}$}] (Yt) at (1.2,2.375) {};
    
    \path[-] (Xt) edge[draw=blue!20, line width=20pt] (Yt);
    \path[-] (Xt) edge[draw=black!20, line width=20pt] (Yj);
    \path[-] (Xj) edge[draw=blue!20, line width=20pt] (Yj);
    \path[-] (Xj) edge[draw=black!20, line width=20pt] (Yi);
    
    \path[-] (v1) edge[draw=red] (x11);
    \path[-] (v1) edge[draw=red] (x21);
    \path[-] (v1) edge[draw=red] (x12);
    \path[-] (v1) edge[draw=red] (x22);
    
    \path[-] (y11) edge[draw=black] (x11);
    \path[-] (y12) edge[draw=black] (x11);
    \path[-] (y11) edge[draw=black] (x12);
    \path[-] (y12) edge[draw=black] (x12);
    
    \path[-] (y21) edge[draw=black] (x11);
    \path[-] (y21) edge[draw=black] (x12);
    \path[-] (y21) edge[draw=black] (x21);
    \path[-] (y21) edge[draw=black] (x22);
    
    \path[-] (y22) edge[draw=black] (x31);
    \path[-] (y22) edge[draw=black] (x32);
    \path[-] (y22) edge[draw=black] (x21);
    \path[-] (y22) edge[draw=black] (x22);
    
    \path[-] (yj) edge[draw=black] (x31);
    \path[-] (yj) edge[draw=black] (x32);
    \path[-] (yj) edge[draw=black] (x21);
    \path[-] (yj) edge[draw=black] (x22);
    
    \path[-] (xj1) edge[draw=black] (yj1);
    \path[-] (xj1) edge[draw=black] (yj2);
    \path[-] (xj1) edge[draw=black] (yj3);
    
    \path[-] (xj2) edge[draw=black] (yj1);
    \path[-] (xj2) edge[draw=black] (yj2);
    \path[-] (xj2) edge[draw=black] (yj3);
    
    \path[-] (xj3) edge[draw=black] (yj4);
    \path[-] (xj3) edge[draw=black] (yj5);
    \path[-] (xj3) edge[draw=black] (yj3);
    
    \path[-] (xj4) edge[draw=black] (yj4);
    \path[-] (xj4) edge[draw=black] (yj5);
    \path[-] (xj4) edge[draw=black] (yj3);
    
    \path[-] (yi) edge[draw=black] (xj1);
    \path[-] (yi) edge[draw=black] (xj2);
    \path[-] (yi) edge[draw=black] (xj3);
    \path[-] (yi) edge[draw=black] (xj4);

     \end{tikzpicture}
    
    \caption{Absorbing vertex $v_t$ in the case $r=2s=4$. The blue connection between classes indicates that those edges belong to the fixed matching in the cluster graph while gray connections indicate using additional edges.}
    \label{fig:Step52}
\end{figure}

We continue in a similar fashion to cover $v_2,\dots,v_s$.
For this let $t=2,\dots,s$ and assume that $i_{t-1}$ is such that $|Y_{i_{t-1}}|-t=|X_{i_{t-1}}| \equiv 0 \pmod s$.
Let $\cA$ be the neighbours of $X_{i_{t-1}}$ in $R$ and let $\cB$ be those clusters which share an edge of $M$ with a cluster from $\cA$.
By Fact~\ref{fact:non-extremal} applied to $\cB$ and $H(v_t)$ there are indices $i_t,j$ such that $X_{i_t} \in H(v_t)$, $Y_{i_t},X_j,Y_j \in H$ and $Y_{i_t}Y_j$ and $X_jX_{i_{t-1}}$ are edges of $R$.
This gives the path $Y_{i_{t-1}},X_J,Y_j,X_{i_t},Y_{i_t}$ in $R$.
We extend the paths $P_{i_t}$ and $P_j$, such that after removing the internal vertices we have $|Y_{i_{t-1}}|=|X_{i_{t-1}}| \equiv 0 \pmod s$, $|Y_j|=|X_j|\equiv 0 \pmod s$, and $|Y_{i_{t}}|-(t+1)=|X_{i_{t}}|\equiv 0 \pmod s$.
This can be done similarly as for $P_{i_1}$ above and we add the internal vertices to $W_0$ (this adds $6s$ vertices, see Figure~\ref{fig:Step52}).

We repeat the same procedure to cover $v_{2s},\dots,v_{s+1}$ and are left do deal with $|Y_{i_{s}}|-s=|X_{i_{s}}|\equiv 0 \pmod s$ and $|Y_{i_{s+1}}|-s=|X_{i_{s+1}}|\equiv 0 \pmod s$.
For this let $\cA$ be the clusters in $R$ that share an edge in $M$ with a neighbour of the cluster $Y_{i_j}$ and, similarly, let $\cB$ be the clusters in $R$ that share an edge in $M$ with a neighbour of the cluster $Y_{i_{j+1}}$.
By Fact~\ref{fact:non-extremal} applied to $\cA$ and $\cB$ we find indices $j_1,j_2$ such that $X_{j_1},Y_{j_1},X_{j_2},Y_{j_2} \in H$ and $Y_{i_{s}}X_{j_1}$, $Y_{j_1}Y_{j_2}$, and $X_{j_2}Y_{i_{s+1}}$ are edges of $R$.
We extend the paths $P_{j_1}$ and $P_{j_2}$ such that after removing the internal vertices we have $|X_{j_1}|=|Y_{j_1}|\equiv 0 \pmod s$, $|X_{j_2}|=|Y_{j_2}|\equiv 0 \pmod s$, $|Y_{i_{s}}|=|X_{i_{s}}|\equiv 0 \pmod s$, and $|Y_{i_{s+1}}|=|X_{i_{s+1}}|\equiv 0 \pmod s$.
We add the internal vertices to $W_0$ (this adds $10s$ vertices, see Figure~\ref{fig:Step53}).

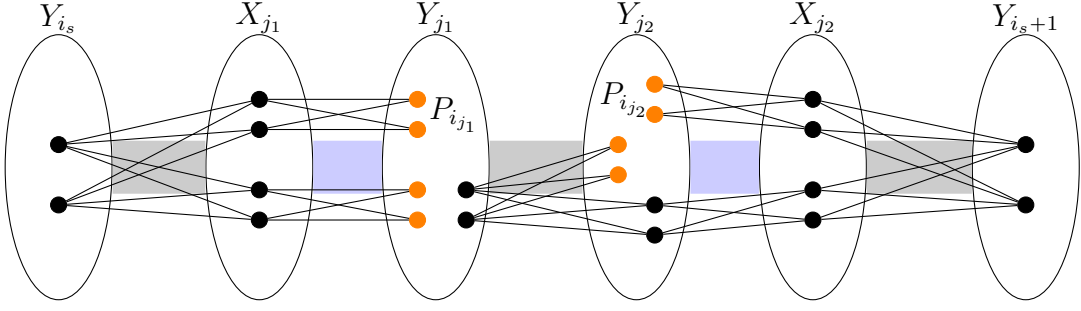
\begin{figure}
    \centering
    \begin{tikzpicture}[scale=0.8, every node/.style={inner sep=0.2,outer sep=0.2}]

    \node[circle, draw=black,fill=black, minimum width = 6,  font=\tiny] (yi1) at (-11.2, 1.75) {};
    \node[circle, draw=black,fill=black, minimum width = 6,  font=\tiny] (yi2) at (-11.2, 2.75) {};
    
    \node[circle, draw=black,fill=black, minimum width = 6,  font=\tiny] (xj11) at (-7.9, 1.5) {};
    \node[circle, draw=black,fill=black, minimum width = 6,  font=\tiny] (xj12) at (-7.9, 2) {};
    \node[circle, draw=black,fill=black, minimum width = 6,  font=\tiny] (xj13) at (-7.9, 3) {};
    \node[circle, draw=black,fill=black, minimum width = 6,  font=\tiny] (xj14) at (-7.9, 3.5) {};
    
    \node[circle, draw=orange,fill=orange, minimum width = 6,  font=\tiny] (yj11) at (-5.3, 1.5) {};
    \node[circle, draw=orange,fill=orange, minimum width = 6,  font=\tiny] (yj12) at (-5.3, 2) {};
    \node[circle, draw=orange,fill=orange, minimum width = 6,  font=\tiny] (yj13) at (-5.3, 3) {};
    \node[circle, draw=orange,fill=orange, minimum width = 6,  font=\tiny] (yj14) at (-5.3, 3.5) {};
    \node[] at (-4.7, 3.25) {$P_{i_{j_1}}$};
    
    \node[circle, draw=black,fill=black, minimum width = 6,  font=\tiny] (yj15) at (-4.5, 1.5) {};
    \node[circle, draw=black,fill=black, minimum width = 6,  font=\tiny] (yj16) at (-4.5, 2) {};

    \node[circle, draw=black,fill=black, minimum width = 6,  font=\tiny] (yj21) at (-1.4, 1.25) {};
    \node[circle, draw=black,fill=black, minimum width = 6,  font=\tiny] (yj22) at (-1.4, 1.75) {};
    \node[circle, draw=orange,fill=orange, minimum width = 6,  font=\tiny] (yj23) at (-2, 2.25) {};
    \node[circle, draw=orange,fill=orange, minimum width = 6,  font=\tiny] (yj24) at (-2, 2.75) {};
    \node[circle, draw=orange,fill=orange, minimum width = 6,  font=\tiny] (yj25) at (-1.4, 3.25) {};
    \node[circle, draw=orange,fill=orange, minimum width = 6,  font=\tiny] (yj26) at (-1.4, 3.75) {};
    \node[] at (-1.9, 3.5) {$P_{i_{j_2}}$};
    
    \node[circle, draw=black,fill=black, minimum width = 6,  font=\tiny] (xj21) at (1.2, 1.5) {};
    \node[circle, draw=black,fill=black, minimum width = 6,  font=\tiny] (xj22) at (1.2, 2) {};
    \node[circle, draw=black,fill=black, minimum width = 6,  font=\tiny] (xj23) at (1.2, 3) {};
    \node[circle, draw=black,fill=black, minimum width = 6,  font=\tiny] (xj24) at (1.2, 3.5) {};
    
    \node[circle, draw=black,fill=black, minimum width = 6,  font=\tiny] (ys1) at (4.7, 1.75) {};
    \node[circle, draw=black,fill=black, minimum width = 6,  font=\tiny] (ys2) at (4.7, 2.75) {};

    \node[ellipse, draw = black, minimum width = 40, minimum height = 100, label = above:{$Y_{i_{s}}$}] (Yi) at (-11.2,2.375) {};
    
    \node[ellipse, draw = black, minimum width = 40, minimum height = 100, label = above:{$X_{j_1}$}] (Xj1) at (-7.9,2.375) {};
    
    \node[ellipse, draw = black, minimum width = 40, minimum height = 100, label = above:{$Y_{j_1}$}] (Yj1) at (-5,2.375) {};
    
    \node[ellipse, draw = black, minimum width = 40, minimum height = 100, label = above:{$Y_{j_2}$}] (Yj2) at (-1.7,2.375) {};
    
    \node[ellipse, draw = black, minimum width = 40, minimum height = 100, label = above:{$X_{j_2}$}] (Xj2) at (1.2,2.375) {};
    
    \node[ellipse, draw = black, minimum width = 40, minimum height = 100, label = above:{$Y_{i_s+1}$}] (Ys) at (4.7,2.375) {};
    
    \path[-] (Yi) edge[draw=black!20, line width=20pt] (Xj1);
    \path[-] (Xj1) edge[draw=blue!20, line width=20pt] (Yj1);
    \path[-] (Yj1) edge[draw=black!20, line width=20pt] (Yj2);
    \path[-] (Yj2) edge[draw=blue!20, line width=20pt] (Xj2);
    \path[-] (Xj2) edge[draw=black!20, line width=20pt] (Ys);
    
    \path[-] (yi1) edge[draw=black] (xj11);
    \path[-] (yi1) edge[draw=black] (xj12);
    \path[-] (yi1) edge[draw=black] (xj13);
    \path[-] (yi1) edge[draw=black] (xj14);
    
    \path[-] (yi2) edge[draw=black] (xj11);
    \path[-] (yi2) edge[draw=black] (xj12);
    \path[-] (yi2) edge[draw=black] (xj13);
    \path[-] (yi2) edge[draw=black] (xj14);
    
    \path[-] (xj11) edge[draw=black] (yj11);
    \path[-] (xj12) edge[draw=black] (yj11);
    \path[-] (xj11) edge[draw=black] (yj12);
    \path[-] (xj12) edge[draw=black] (yj12);
    
    \path[-] (xj13) edge[draw=black] (yj13);
    \path[-] (xj14) edge[draw=black] (yj13);
    \path[-] (xj13) edge[draw=black] (yj14);
    \path[-] (xj14) edge[draw=black] (yj14);
    
    \path[-] (yj15) edge[draw=black] (yj21);
    \path[-] (yj15) edge[draw=black] (yj22);
    \path[-] (yj15) edge[draw=black] (yj23);
    \path[-] (yj15) edge[draw=black] (yj24);
    
    \path[-] (yj16) edge[draw=black] (yj21);
    \path[-] (yj16) edge[draw=black] (yj22);
    \path[-] (yj16) edge[draw=black] (yj23);
    \path[-] (yj16) edge[draw=black] (yj24);
    
    \path[-] (yj21) edge[draw=black] (xj21);
    \path[-] (yj21) edge[draw=black] (xj22);
    \path[-] (yj22) edge[draw=black] (xj21);
    \path[-] (yj22) edge[draw=black] (xj22);
    
    \path[-] (yj25) edge[draw=black] (xj23);
    \path[-] (yj25) edge[draw=black] (xj24);
    \path[-] (yj26) edge[draw=black] (xj23);
    \path[-] (yj26) edge[draw=black] (xj24);

    \path[-] (xj21) edge[draw=black] (ys1);
    \path[-] (xj21) edge[draw=black] (ys2);
    \path[-] (xj22) edge[draw=black] (ys1);
    \path[-] (xj22) edge[draw=black] (ys2);
    \path[-] (xj23) edge[draw=black] (ys1);
    \path[-] (xj23) edge[draw=black] (ys2);
    \path[-] (xj24) edge[draw=black] (ys1);
    \path[-] (xj24) edge[draw=black] (ys2);

     \end{tikzpicture}
    
    \caption{Balancing $Y_{i_{s}}$ and $Y_{i_{s+1}}$ in the case $r=2s=4$. As before, we indicate the fixed matching in the cluster graph by blue connections and additional edges by gray connections.}
    \label{fig:Step53}
\end{figure}

It is easy to ensure that in each of these steps the new end-tuples are always well-connected.
During this procedure when absorbing $2s$ vertices into the paths we added at most $18 s^2$ vertices to $W_0$.

\smallskip
\textbf{Reset after $\nu n$ iterations.}
We can repeat this for $\nu n$ steps as this still guarantees that $|W_0| \le 18 s^2 \nu n$ as required.
It follows from the bound on $\nu$ that no degree in the regular pairs in $M$ dropping below $\tfrac d2 \cdot \tfrac{n}{\ell}$.
Moreover, for $i \in [\ell']$ the end-tuples of the path $P_i$ that were well-connected at some point still have at least $\tfrac 14 d^s L$ common neighbours into the respective sets.
After $\nu n$ steps we want to get close enough to the original situation such that we can continue for another $\nu n$ steps.
By Lemma~\ref{lem:superreg} for any $i \in [\ell']$ we need to remove at most $2 \eps \tfrac n\ell$ vertices from the $X_i,Y_i$ to get that  $(X_i,Y_i)$ is $(2\eps,\tfrac{3d}{4})$-super-regular.
We will greedily absorb these vertices into the path $P_i$ by alternating between $X_i$ and $Y_i$ without any degree dropping below $\tfrac{3d}{4} \cdot \tfrac n\ell$.
Here we use that the original pair was $(\eps,d)$-regular and that it intersects $W$ in at most $\tfrac 12 \tfrac n\ell$ vertices.
While doing this, we can ensure that, for each $i \in [\ell']$, both ends of the path $P_i$ are extended at least two steps and the ends are well-connected again.
We add the vertices used in these paths to $W$, also move the vertices from $W_0$ to $W$, and set $W_0=\emptyset$.
We continue by covering $2s$ vertices from $V_0$ as explained above.

\smallskip
\textbf{Covering the last vertices.}
We can repeat this until $|V_0|<2s$ and if $|V_0|=0$ we are done with this step.
Otherwise, we have $|V_0|=t\not=0$, $n \not\equiv 0 \pmod{2s}$, and we can not find the $\tfrac r2$-blow-up of a cycle.
If $r$ is odd $nr \equiv 0 \pmod 2$ implies that $n$ and also $t$ are even.
We need to absorb the last $t$ vertices in a different way.
If $r$ is even, let $v \in V_0$ and with Fact~\ref{fact:non-extremal} pick $j_1,j_4$ such that $X_{j_1},X_{j_4} \in H(v)$, $Y_{j_1},Y_{j_4} \in H$, and $X_{j_1}X_{j_4}$ is an edge of $R$.
Then we consider those clusters that share an edge of $M$ with a neighbour of $Y_{j_1}$ and $Y_{j_4}$ respectively and with Fact~\ref{fact:non-extremal} pick $j_2,j_3$ such that $X_{j_2},Y_{j_2},X_{j_3},Y_{j_3} \in H$, and $Y_{j_1}Y_{j_2}$, $Y_{j_4}Y_{j_3}$, and $X_{j_2}X_{j_3}$ are edges of $R$.
We extend the path $P_{j_1}$ by following $Y_{j_1},X_{j_1},X_{j_4},Y_{j_4},Y_{j_3},X_{j_3},X_{j_2},Y_{j_2},Y_{j_1}$ such that the vertex $v$ is in the neighbourhood of the new vertices from $X_{j_1}$ and $X_{j_4}$.
This allows us to include $v$ into the path.
Note that this is no longer the $\tfrac r2$-blow-up of a path (see Figure~\ref{fig:Step54}).
If $r$ is odd let $u,v \in V_0$ and we proceed similarly to include both vertices (see Figure~\ref{fig:Step55}).
Here we find $X_{j_1},X_{j_4} \in H(v)$ and $X_{i_2},X_{j_3} \in H(u)$ such that $X_{j_1}X_{j_4}$ and $X_{i_2}X_{j_3}$ are edges of $R$ and then connect $Y_{j_1}$ to $Y_{j_2}$ and $Y_{j_4}$ to $Y_{j_3}$ as in the even case by using four additional clusters for each connection.
Note that here the path structure on the lower half also is not an $\tfrac r2$-blow-up (alternating $K_{s,s}$ and $K_{s,s}-K_{1,1}^{(s)}$), but has one edge shifted (alternating $K_{s,s}-K_{1,1}$ and $K_{s,s}-K_{1,1}^{(s-1)}$).

\begin{figure}
    \centering
    \begin{tikzpicture}[scale=0.8, every node/.style={inner sep=0.2,outer sep=0.2}]

    \node[circle, draw=orange,fill=orange, minimum width = 6,  font=\tiny] (yj11) at (-5.25, 1) {};
    \node[circle, draw=orange,fill=orange, minimum width = 6,  font=\tiny] (yj12) at (-5.25, 1.5) {};
    \node[circle, draw=orange,fill=orange, minimum width = 6,  font=\tiny] (yj13) at (-4.75, 3) {};
    \node[circle, draw=orange,fill=orange, minimum width = 6,  font=\tiny] (yj14) at (-4.75, 3.5) {};
    \node[] at (-5.25, 3.25) {$P_{i_{j_1}}$};
    
    \node[circle, draw=black,fill=black, minimum width = 6,  font=\tiny] (yj21) at (-1.7, 2.25) {};
    \node[circle, draw=black,fill=black, minimum width = 6,  font=\tiny] (yj22) at (-1.7, 2.75) {};
    
    \node[circle, draw=black,fill=black, minimum width = 6,  font=\tiny] (xj21) at (1.2, 2.25) {};
    \node[circle, draw=black,fill=black, minimum width = 6,  font=\tiny] (xj22) at (1.2, 2.75) {};
    
    \node[circle, draw=black,fill=black, minimum width = 6,  font=\tiny] (xj31) at (1.2, -3.75) {};
    \node[circle, draw=black,fill=black, minimum width = 6,  font=\tiny] (xj32) at (1.2, -4.25) {};
    
    \node[circle, draw=black,fill=black, minimum width = 6,  font=\tiny] (yj31) at (-1.7, -3.75) {};
    \node[circle, draw=black,fill=black, minimum width = 6,  font=\tiny] (yj32) at (-1.7, -4.25) {};
    
    \node[circle, draw=black,fill=black, minimum width = 6,  font=\tiny] (yj41) at (-5, -3.75) {};
    \node[circle, draw=black,fill=black, minimum width = 6,  font=\tiny] (yj42) at (-5, -4.25) {};
    
    \node[circle, draw=black,fill=black, minimum width = 6,  font=\tiny] (xj41) at (-7.9, -3.75) {};
    \node[circle, draw=black,fill=black, minimum width = 6,  font=\tiny] (xj42) at (-7.9, -4.25) {};
    
    \node[circle, draw=black,fill=black, minimum width = 6,  font=\tiny] (xj11) at (-7.9, 2.25) {};
    \node[circle, draw=black,fill=black, minimum width = 6,  font=\tiny] (xj12) at (-7.9, 2.75) {};

    \node[ellipse, draw = black, minimum width = 40, minimum height = 100, label = above:{$X_{j_1}$}] (Xj1) at (-7.9,2.375) {};
    
    \node[ellipse, draw = black, minimum width = 40, minimum height = 100, label = above:{$Y_{j_1}$}] (Yj1) at (-5,2.375) {};
    
    \node[ellipse, draw = black, minimum width = 40, minimum height = 100, label = above:{$Y_{j_2}$}] (Yj2) at (-1.7,2.375) {};
    
    \node[ellipse, draw = black, minimum width = 40, minimum height = 100, label = above:{$X_{j_2}$}] (Xj2) at (1.2,2.375) {};

    \node[ellipse, draw = black, minimum width = 40, minimum height = 100, label = below:{$X_{j_4}$}] (Xj4) at (-7.9,-4) {};
    
    \node[ellipse, draw = black, minimum width = 40, minimum height = 100, label = below:{$Y_{j_4}$}] (Yj4) at (-5,-4) {};
    
    \node[ellipse, draw = black, minimum width = 40, minimum height = 100, label = below:{$Y_{j_3}$}] (Yj3) at (-1.7,-4) {};
    
    \node[ellipse, draw = black, minimum width = 40, minimum height = 100, label = below:{$X_{j_3}$}] (Xj3) at (1.2,-4) {};

    \path[-] (Xj1) edge[draw=blue!20, line width=20pt] (Yj1);
    \path[-] (Yj1) edge[draw=black!20, line width=20pt] (Yj2);
    \path[-] (Yj2) edge[draw=blue!20, line width=20pt] (Xj2);
    
    \path[-] (Xj4) edge[draw=blue!20, line width=20pt] (Yj4);
    \path[-] (Yj4) edge[draw=black!20, line width=20pt] (Yj3);
    \path[-] (Yj3) edge[draw=blue!20, line width=20pt] (Xj3);
    
    \path[-] (Xj1) edge[draw=black!20, line width=20pt] (Xj4);
    \path[-] (Xj2) edge[draw=black!20, line width=20pt] (Xj3);
    
    \node[circle, draw=red,fill=red, minimum width = 6,  font=\tiny] (v) at (-10, -0.875) {};
    \node[] at (-10.5, -0.875) {$v$};
    
    \path[-] (v) edge[draw=red] (xj11);
    \path[-] (v) edge[draw=red] (xj12);
    \path[-] (v) edge[draw=red] (xj41);
    \path[-] (v) edge[draw=red] (xj42);
    
    \path[-] (xj11) edge[draw=black] (xj41);
    \path[-] (xj12) edge[draw=black, bend left = 25] (xj42);
    
    \path[-] (xj21) edge[draw=black, bend right = 20] (xj31);
    \path[-] (xj21) edge[draw=black, bend right = 25] (xj32);
    \path[-] (xj22) edge[draw=black, bend left = 20] (xj31);
    \path[-] (xj22) edge[draw=black, bend left = 25] (xj32);
    
    \path[-] (xj11) edge[draw=black] (yj11);
    \path[-] (xj12) edge[draw=black] (yj11);
    \path[-] (xj11) edge[draw=black] (yj12);
    \path[-] (xj12) edge[draw=black] (yj12);
    
    \path[-] (xj41) edge[draw=black] (yj41);
    \path[-] (xj42) edge[draw=black] (yj41);
    \path[-] (xj41) edge[draw=black] (yj42);
    \path[-] (xj42) edge[draw=black] (yj42);
    
    \path[-] (yj31) edge[draw=black] (yj41);
    \path[-] (yj32) edge[draw=black] (yj41);
    \path[-] (yj31) edge[draw=black] (yj42);
    \path[-] (yj32) edge[draw=black] (yj42);
    
    \path[-] (yj31) edge[draw=black] (xj31);
    \path[-] (yj32) edge[draw=black] (xj31);
    \path[-] (yj31) edge[draw=black] (xj32);
    \path[-] (yj32) edge[draw=black] (xj32);
    
    \path[-] (yj21) edge[draw=black] (xj21);
    \path[-] (yj22) edge[draw=black] (xj21);
    \path[-] (yj21) edge[draw=black] (xj22);
    \path[-] (yj22) edge[draw=black] (xj22);
    
    \path[-] (yj21) edge[draw=black] (yj13);
    \path[-] (yj21) edge[draw=black] (yj14);
    \path[-] (yj22) edge[draw=black] (yj13);
    \path[-] (yj22) edge[draw=black] (yj14);

     \end{tikzpicture}
    
    \caption{Absorbing $v$ in the even case, where $r=2s=4$. The blue and gray connections represent the matching edges and non-matching edges in the cluster graph again.}
    \label{fig:Step54}
\end{figure}
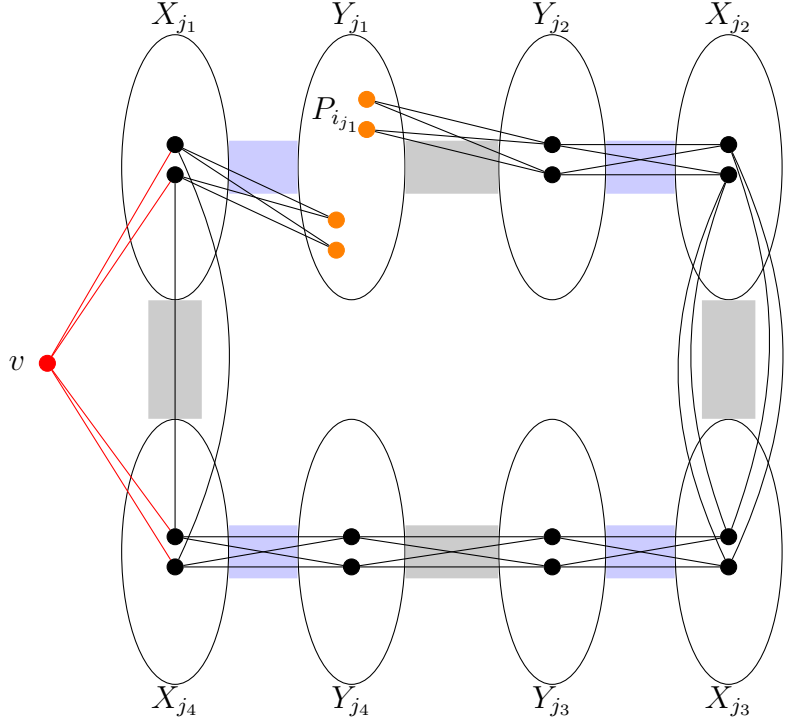

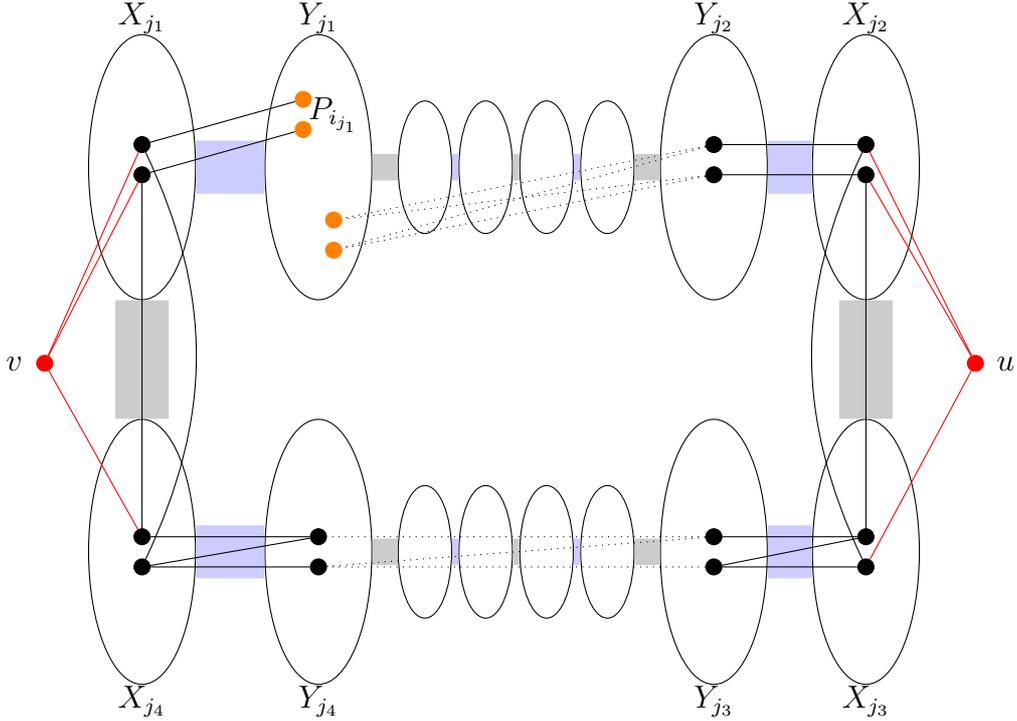
\begin{figure}
    \centering
    \begin{tikzpicture}[scale=0.8, every node/.style={inner sep=0.2,outer sep=0.2}]

    \node[circle, draw=orange,fill=orange, minimum width = 6,  font=\tiny] (yj11) at (-4.75, 1) {};
    \node[circle, draw=orange,fill=orange, minimum width = 6,  font=\tiny] (yj12) at (-4.75, 1.5) {};
    \node[circle, draw=orange,fill=orange, minimum width = 6,  font=\tiny] (yj13) at (-5.25, 3) {};
    \node[circle, draw=orange,fill=orange, minimum width = 6,  font=\tiny] (yj14) at (-5.25, 3.5) {};
    \node[] at (-4.75, 3.25) {$P_{i_{j_1}}$};
    
    \node[circle, draw=black,fill=black, minimum width = 6,  font=\tiny] (yj21) at (1.5, 2.25) {};
    \node[circle, draw=black,fill=black, minimum width = 6,  font=\tiny] (yj22) at (1.5, 2.75) {};
    
    \node[circle, draw=black,fill=black, minimum width = 6,  font=\tiny] (xj21) at (4, 2.25) {};
    \node[circle, draw=black,fill=black, minimum width = 6,  font=\tiny] (xj22) at (4, 2.75) {};
    
    \node[circle, draw=black,fill=black, minimum width = 6,  font=\tiny] (xj31) at (4, -3.75) {};
    \node[circle, draw=black,fill=black, minimum width = 6,  font=\tiny] (xj32) at (4, -4.25) {};
    
    \node[circle, draw=black,fill=black, minimum width = 6,  font=\tiny] (yj31) at (1.5, -3.75) {};
    \node[circle, draw=black,fill=black, minimum width = 6,  font=\tiny] (yj32) at (1.5, -4.25) {};
    
    \node[circle, draw=black,fill=black, minimum width = 6,  font=\tiny] (yj41) at (-5, -3.75) {};
    \node[circle, draw=black,fill=black, minimum width = 6,  font=\tiny] (yj42) at (-5, -4.25) {};
    
    \node[circle, draw=black,fill=black, minimum width = 6,  font=\tiny] (xj41) at (-7.9, -3.75) {};
    \node[circle, draw=black,fill=black, minimum width = 6,  font=\tiny] (xj42) at (-7.9, -4.25) {};
    
    \node[circle, draw=black,fill=black, minimum width = 6,  font=\tiny] (xj11) at (-7.9, 2.25) {};
    \node[circle, draw=black,fill=black, minimum width = 6,  font=\tiny] (xj12) at (-7.9, 2.75) {};

    \node[ellipse, draw = black, minimum width = 40, minimum height = 100, label = above:{$X_{j_1}$}] (Xj1) at (-7.9,2.375) {};
    
    \node[ellipse, draw = black, minimum width = 40, minimum height = 100, label = above:{$Y_{j_1}$}] (Yj1) at (-5,2.375) {};
    
    \node[ellipse, draw = black, minimum width = 40, minimum height = 100, label = above:{$Y_{j_2}$}] (Yj2) at (1.5,2.375) {};
    
    \node[ellipse, draw = black, minimum width = 40, minimum height = 100, label = above:{$X_{j_2}$}] (Xj2) at (4,2.375) {};

    \node[ellipse, draw = black, minimum width = 40, minimum height = 100, label = below:{$X_{j_4}$}] (Xj4) at (-7.9,-4) {};
    
    \node[ellipse, draw = black, minimum width = 40, minimum height = 100, label = below:{$Y_{j_4}$}] (Yj4) at (-5,-4) {};
    
    \node[ellipse, draw = black, minimum width = 20, minimum height = 50] (S1) at (-3.25,-4) {};
    \node[ellipse, draw = black, minimum width = 20, minimum height = 50] (S2) at (-2.25,-4) {};
    \node[ellipse, draw = black, minimum width = 20, minimum height = 50] (S3) at (-1.25,-4) {};
    \node[ellipse, draw = black, minimum width = 20, minimum height = 50] (S4) at (-0.25,-4) {};
    
    \node[ellipse, draw = black, minimum width = 20, minimum height = 50] (T1) at (-3.25,2.375) {};
    \node[ellipse, draw = black, minimum width = 20, minimum height = 50] (T2) at (-2.25,2.375) {};
    \node[ellipse, draw = black, minimum width = 20, minimum height = 50] (T3) at (-1.25,2.375) {};
    \node[ellipse, draw = black, minimum width = 20, minimum height = 50] (T4) at (-0.25,2.375) {};
    
    \node[ellipse, draw = black, minimum width = 40, minimum height = 100, label = below:{$Y_{j_3}$}] (Yj3) at (1.5,-4) {};
    
    \node[ellipse, draw = black, minimum width = 40, minimum height = 100, label = below:{$X_{j_3}$}] (Xj3) at (4,-4) {};

    \path[-] (Xj1) edge[draw=blue!20, line width=20pt] (Yj1);

    \path[-] (Yj1) edge[draw=black!20, line width=10pt] (T1);
    \path[-] (T1) edge[draw=blue!20, line width=10pt] (T2);
    \path[-] (T2) edge[draw=black!20, line width=10pt] (T3);
    \path[-] (T3) edge[draw=blue!20, line width=10pt] (T4);
    \path[-] (Yj2) edge[draw=black!20, line width=10pt] (T4);

    \path[-] (Yj2) edge[draw=blue!20, line width=20pt] (Xj2);
    
    \path[-] (Xj4) edge[draw=blue!20, line width=20pt] (Yj4);
    
    \path[-] (Yj4) edge[draw=black!20, line width=10pt] (S1);
    \path[-] (S1) edge[draw=blue!20, line width=10pt] (S2);
    \path[-] (S2) edge[draw=black!20, line width=10pt] (S3);
    \path[-] (S3) edge[draw=blue!20, line width=10pt] (S4);
    \path[-] (Yj3) edge[draw=black!20, line width=10pt] (S4);

    \path[-] (Yj3) edge[draw=blue!20, line width=20pt] (Xj3);
    
    \path[-] (Xj1) edge[draw=black!20, line width=20pt] (Xj4);
    \path[-] (Xj2) edge[draw=black!20, line width=20pt] (Xj3);
    
    \node[circle, draw=red,fill=red, minimum width = 6,  font=\tiny] (v) at (-9.5, -0.875) {};
    \node[] at (-10, -0.875) {$v$};
    \node[circle, draw=red,fill=red, minimum width = 6,  font=\tiny] (u) at (5.8, -0.875) {};
    \node[] at (6.3, -0.875) {$u$};
    
    \path[-] (v) edge[draw=red] (xj11);
    \path[-] (v) edge[draw=red] (xj12);
    \path[-] (v) edge[draw=red] (xj41);

    \path[-] (u) edge[draw=red] (xj32);
    \path[-] (u) edge[draw=red] (xj21);
    \path[-] (u) edge[draw=red] (xj22);
    
    \path[-] (xj11) edge[draw=black] (xj41);
    \path[-] (xj12) edge[draw=black, bend left = 25] (xj42);
    
    \path[-] (xj21) edge[draw=black] (xj31);
    \path[-] (xj22) edge[draw=black, bend right = 25] (xj32);
    
    \path[-] (xj12) edge[draw=black] (yj14);
    \path[-] (xj11) edge[draw=black] (yj13);

    \path[-] (xj41) edge[draw=black] (yj41);
    \path[-] (xj42) edge[draw=black] (yj41);
    \path[-] (xj42) edge[draw=black] (yj42);
    
    \path[-] (yj31) edge[draw=black, dotted] (yj41);
    \path[-] (yj31) edge[draw=black, dotted] (yj42);
    \path[-] (yj32) edge[draw=black, dotted] (yj42);
    
    \path[-] (yj31) edge[draw=black] (xj31);
    \path[-] (yj32) edge[draw=black] (xj31);
    \path[-] (yj32) edge[draw=black] (xj32);
    
    \path[-] (yj21) edge[draw=black] (xj21);
    \path[-] (yj22) edge[draw=black] (xj22);
    
    \path[-] (yj21) edge[draw=black, dotted] (yj11);
    \path[-] (yj21) edge[draw=black, dotted] (yj12);
    \path[-] (yj22) edge[draw=black, dotted] (yj11);
    \path[-] (yj22) edge[draw=black, dotted] (yj12);

     \end{tikzpicture}
    
    \caption{Absorbing $v$ and $u$ in the odd case, where $r=3, s=2$. The blue and gray connections represent the matching edges and non-matching edges in the cluster graph again. The dotted edges indicate an $r-$regular $r-$connected path.}
    \label{fig:Step55}
\end{figure}

\smallskip
\textbf{Summary.}
We need to estimate the number of vertices added to $W$ throughout the whole procedure of covering $V_0$, which are at most $2 \beta n$ vertices.
There are at most $\lceil \tfrac{2\beta n}{2s\nu n} \rceil \le \tfrac{\beta}{s \nu} + 1$ iterations of the argument for covering $2s\nu n$ vertices and for covering $2s$ vertices of $V_0$ we need at most $18 s^2$ vertices.
During these iterations we will always have 
\[|W| \le \frac{18s^2}{2s} 2 \beta n + \left(\frac{\beta}{s\nu}+1\right) \cdot \ell' 2 \eps \frac n\ell \cdot 2s \le 20 s \beta n\,,\]
where the second term comes from the vertices we need to absorb after each iteration.
Covering the last $t<2s$ vertices does not change anything and, therefore, we can indeed repeat this until all vertices from $V_0$ are covered by the paths.

\smallskip
\noindent \textbf{\ref{step:spanning}.}
We fully absorbed $V_0$ into the connecting paths such that the end-tuples are well-connected to the other side of the matching edge.
Let $i \in [\ell']$ and denote by $\mathbf{x}=(x_1,\dots, x_s)$ and $\mathbf{y}=(y_1,\dots, y_s)$ the end-tuples of the paths $P_{i-1}$ and $P_{i}$, respectively.
Remove $\mathbf{x}$ from $X_i$ and $\mathbf{y}$ from $Y_i$, note that $|X_i|=|Y_i| \equiv 0 \pmod s$ and that $(X_i,Y_i)$ is $(2\eps,\tfrac d2)$-super-regular.
Denote the common neighbours of $\mathbf{x}$ in $Y_i$ by $Y'$, the common neighbours of $\mathbf{y}$ in $X_i$ by $X'$, and note that $|X'|,|Y'| \ge \tfrac 14 d^s |X|$.
Therefore, we can apply Lemma~\ref{lem:blowup} to cover $X_i$ and $Y_i$ with the $\tfrac r2$-blow-up of a path and end-tuples within $X'$ and $Y'$, which then connects $P_{i-1}$ to $P_{i}$.

Together this gives an $r$-regular subgraph in $G$.
In the case when $n \equiv 0 \pmod{2s}$ we have constructed the $\tfrac r2$-blow-up of a path, which is $r$-connected.
To see that it is also $r$-connected in the other cases it suffices to observe that in the case when $r$ is even removing a perfect matching from a $K_{s,s}$ and adding a vertex $v$ to all these $r$ vertices (see Figure~\ref{fig:Step54}) preserves this property.
Similarly, in the case when $r$ is odd, removing a perfect matching from two copies of $K_{s,s}$, connecting vertices $u,v$ two $r$ of these vertices, and shifting the path in between as described (see Figure~\ref{fig:Step55}) also preserves this property.
\end{proof}

\bibliography{bibliography} 

\begin{thebibliography}{10}

\bibitem{dirac}
G.~A. Dirac, ``Some theorems on abstract graphs,'' {\em Proceedings of the
  London Mathematical Society}, vol.~s3-2, no.~1, pp.~69--81, 1952.

\bibitem{KSS_Seynmour}
J.~Koml{\'{o}}s, G.~N. S{\'{a}}rk\"{o}zy, and E.~Szemer{\'{e}}di, ``Proof of
  the seymour conjecture for large graphs,'' {\em Annals of Combinatorics},
  vol.~2, pp.~43--60, Mar. 1998.

\bibitem{CH63}
K.~{Corradi} and A.~{Hajnal}, ``{On the maximal number of independent circuits
  in a graph},'' {\em {Acta Math. Acad. Sci. Hung.}}, vol.~14, pp.~423--439,
  1963.

\bibitem{HS_erdos}
A.~Hajnal and E.~Szemerédi, ``{Proof of a conjecture of Erd\H{o}s},'' in {\em
  {Combinatorial Theory and its Applications}}, pp.~601--623, Colloq. Math.
  Soc. J. Bolyai 4, 1970.

\bibitem{BST09}
J.~{B\"ottcher}, M.~{Schacht}, and A.~{Taraz}, ``{Proof of the bandwidth
  conjecture of Bollob\'as and Koml\'os},'' {\em {Math. Ann.}}, vol.~343,
  no.~1, pp.~175--205, 2009.

\bibitem{KO09}
D.~{K\"uhn} and D.~{Osthus}, ``{Embedding large subgraphs into dense graphs},''
  in {\em {Surveys in combinatorics 2009. Papers from the 22nd British
  combinatorial conference, St. Andrews, UK, July 5--10, 2009}}, pp.~137--167,
  Cambridge: Cambridge University Press, 2009.

\bibitem{MKcomm}
M.~Kriesell. {Personal communication}.

\bibitem{BJK19}
J.~Bang-Jensen and M.~Kriesell, ``Good acyclic orientations of 4-regular
  4-connected graphs,'' {\em arXiv preprint arXiv:1912.04569}, 2019.

\bibitem{Bolbook98}
B.~Bollobas, {\em Modern Graph Theory}.
\newblock Berlin Heidelberg: Springer Science \& Business Media, 2013.

\bibitem{Sze_regularity}
E.~{Szemer\'edi}, ``{Regular partitions of graphs},'' {\em {Probl\`emes
  combinatoires et th\'eorie des graphes, Orsay 1976, Colloq. int. CNRS No.
  260, 399-401 (1978).}}, 1978.

\bibitem{KSS_Blowup}
J.~Koml{\'{o}}s, G.~S{\'{a}}rk\"{o}zy, and E.~Szemer{\'{e}}di, ``Blow-up
  lemma,'' {\em Combinatorica}, vol.~17, pp.~109--123, Mar. 1997.

\bibitem{koml_simon}
J.~Koml{\'o}s and M.~Simonovits, ``Szemer{\'e}di's regularity lemma and its
  applications in graph theory,'' in {\em Combinatorics, {P}aul {E}rd\H{o}s is
  eighty, {V}ol. 2 ({K}eszthely, 1993)}, pp.~295--352, J\'{a}nos Bolyai Math.
  Soc., Budapest, 1996.

\bibitem{KSS_square}
J.~Koml\'{o}s, G.~N. S\'{a}rk\"{o}zy, and E.~Szemer\'{e}di, ``On the square of
  a hamiltonian cycle in dense graphs,'' {\em Random Struct. Algorithms},
  vol.~9, p.~193–211, Aug. 1996.

\end{thebibliography}
\bibliographystyle{ieeetr}

\endgroup
\end{document}